\documentclass[12pt,twoside,leqno]{article}
\usepackage{eurosym}
\usepackage[T1]{fontenc}
\usepackage{amsmath,amsthm}
\usepackage{amssymb,latexsym}
\usepackage{enumerate}

\theoremstyle{plain}
\newtheorem{thm}{Theorem}[section]
\newtheorem{lem}[thm]{Lemma}
\newtheorem{prop}[thm]{Proposition}
\newtheorem{cor}[thm]{Corollary}

\theoremstyle{definition}

\newtheorem{defi}[thm]{Definition}
\theoremstyle{remark}
\newtheorem{rem}[thm]{Remark}

\newtheorem{p}[thm]{Problem}

\newcommand{\mc}{\mathcal}
\begin{document}

\title{Hausdorff compactifications in ZF}
\author{Kyriakos Keremedis \\
Department of Mathematics\\
University of Aegean\\
Karlovassi, Samos 83200, Greece\\
E-mail: kker@aegean.gr\\
and\\
Eliza Wajch\\
Institute of Mathematics and Physics\\
Siedlce University of Natural Sciences and Humanities\\
3 Maja 54,\\
08-110 Siedlce, Poland\\
E-mail: eliza.wajch@wp.pl }
\maketitle

\begin{abstract}
For a compactification $\alpha X$ of a Tychonoff space $X$, the
algebra of all functions $f\in C(X)$ that are continuously extendable over $%
\alpha X$ is denoted by $C_{\alpha}(X)$. It is shown that, in a model of 
\textbf{ZF}, it may happen that a discrete space $X$ can have non-equivalent
Hausdorff compactifications $\alpha X$ and $\gamma X$ such that $%
C_{\alpha}(X)=C_{\gamma}(X)$. Amorphous sets are applied to a proof that Glicksberg's theorem that $\beta X\times \beta
Y$ is the \v Cech-Stone compactification of $X\times Y$ when $X\times Y$ is
a Tychonoff pseudocompact space is false in some models of $\mathbf{ZF}$. It is noticed that if all Tychonoff compactifications of locally compact spaces had $C^{\ast}$-embedded remainders, then van Douwen's choice principle would be satisfied.
Necessary and sufficient conditions for a set of continuous bounded real
functions on a Tychonoff space $X$ to generate a compactification of $X$ are
given in $\mathbf{ZF}$. A concept of a functional \v Cech-Stone compactification is investigated in the absence of the axiom of choice. 
\end{abstract}
\renewcommand{\thefootnote}{}



\renewcommand{\thefootnote}{\arabic{footnote}} \setcounter{footnote}{0}

\section{Introduction}

For a topological space $X$, we denote by $C(X)$ the algebra of all
continuous real functions on $X$ and by $C^{\ast }(X)$ the algebra of all
bounded continuous real functions on $X$.  We recall that
a \emph{compactification} of $X$ is an ordered pair $\langle \alpha X,\alpha
\rangle $ such that $\alpha X$ is a compact (not necessarily Hausdorff)
space, while $\alpha :X\rightarrow \alpha X$ is a homeomorphic embedding
such that $\alpha (X)$ is dense in $\alpha X$. Usually, a compactification $%
\langle \alpha X,\alpha \rangle $ of $X$ is denoted by $\alpha X$, the space 
$X$ is identified with $\alpha (X)$ and $\alpha $ is treated as the identity
map $\text{id}_{X}:X\rightarrow \alpha X$. Thus, in abbreviation, we can say
that a compactification of $X$ is a compact space $\alpha X$ such that $X$
is a dense subspace of $\alpha X$. The \emph{remainder of a compactification}
$\alpha X$ of $X$ is the set $\alpha X\setminus X$. We use the notation $%
\alpha X\thickapprox \gamma X$ to say that compactifications $\alpha X$ and $%
\gamma X$ of $X$ are equivalent, i.e. that there exists a homeomorphism $%
h:\alpha X\rightarrow \gamma X$ such that $h\circ \alpha =\gamma $. If there
exists a continuous map $f:\alpha X\rightarrow \gamma X$ such that $f\circ
\alpha =\gamma $, we write $\gamma X\leq \alpha X$. If $\alpha X$ and $%
\gamma X$ are Hausdorff compactifications of $X$, then $\alpha X\thickapprox
\gamma X$ if and only if $\alpha X\leq \gamma X$ and $\gamma X\leq \alpha X$.

Throughout this article, we assume the same system $\mathbf{ZF}$ of
set-theoretic axioms as in \cite{PW} to develop a theory of Hausdorff
compactifications without the axiom of choice \textbf{AC}. Of course, we
assume that $\mathbf{ZF}$ is consistent. We clearly denote theorems provable
in $\mathbf{ZF}$ or, respectively, in $\mathbf{ZF}$ enriched by an additional
axiom $\mathbf{A}$ by putting $\mathbf{[ZF]}$ or $\mathbf{[ZF+A]}$,
respectively, at the beginning of theorems. We use the same notation as in 
\cite{Her} for weaker forms of $\mathbf{AC}$, and we shall refer to a
corresponding form from \cite{HR}. In particular, $\mathbf{UFT}$ stands for
the ultrafilter theorem which states that, for every set $X$, each filter in
the power set $\mathcal{P}(X)$ can be enlarged to an ultrafilter (cf.
Definition 2.15 in \cite{Her} and Form 14 A in \cite{HR}). 

If $F$ is a set of mappings $f: X\to Y_f$, then the \emph{evaluation map} $e_F: X\to\prod_{f\in F} Y_f$ is defined  by: $e_{F}(x)(f)=f(x)$ for any $x\in X$ and $f\in F$ (cf. Definition 1.23 in \cite{Ch}).  For a Tychonoff space $X$, we denote by $\mathcal{E}(X)$ the collection of
all $F\subseteq C^{\ast }(X)$ such that $e_{F}: X\to\mathbb{R}^F$ is a homeomorphic embedding.
If $F\in \mathcal{E}(X)$, we denote by $e_{F}X$ the closure of $e_{F}(X)$ in 
$\mathbb{R}^{F}$. In $\mathbf{ZFC}$-theory of Hausdorff compactifications, it
is well known that every Hausdorff compactification $\alpha X$ of a non-empty
space $X$ is strictly determined by the algebra $C_{\alpha }(X)$ of all
continuous real functions on $X$ that are continuously extendable over $%
\alpha X$; more precisely, $\alpha X$ is equivalent with $e_{F}X$ for $%
F=C_{\alpha }(X)$. Moreover, if $\alpha X$ and $\gamma X$ are Hausdorff
compactifications of $X$, then it holds true in $\mathbf{ZFC}$ that $\alpha X$
is equivalent with $\gamma X$ if and only if $C_{\alpha }(X)=C_{\gamma }(X)$
(cf. Theorem 2.10 in \cite{Ch}). However, in $\mathbf{ZF}$, it is equivalent with $\mathbf{%
UFT}$ that, for every non-empty  Tychonoff space $X$ and for each $F\in 
\mathcal{E}(X)$, the space $e_{F}X$ is compact (see, for instance, Theorem
10.12 of \cite{PW} and Theorem 4.37 in \cite{Her}). In our article,  we pay a special
attention to the fact that there is a model of \textbf{ZF} in which there
are Tychonoff spaces that have Hausdorff, not completely regular
compactifications.

In Section 2, we prove that, in a model of $\mathbf{ZF}$, it may happen that
even a discrete space $X$ can have non-equivalent Hausdorff
compactifications $\alpha X$ and $\gamma X$ such that $C_{\alpha
}(X)=C_{\gamma }(X)$ and $\gamma X$ is not completely regular. A not
completely regular Hausdorff compactification of a Tychonoff space is called 
\emph{a strange compactification}. We show that all Hausdorff
compactifications with finite remainders of Tychonoff spaces are completely
regular in \textbf{ZF}. We prove that it may happen in a model of $\mathbf{ZF}$ that the remainder a Tychonoff compactification of a locally compact space can be not $C^{\ast}$-embedded in the compactification. If for every locally compact space $X$ and for every Tychonoff compactification $\alpha X$ of $X$  the remainder $\alpha X\setminus X$ is $C^{\ast}$-embedded in $\alpha X$, then van Douwen's choice principle (Form 119 in \cite{HR}) must hold.  We notice that an amorphous set exists if and only
if there exists an infinite discrete space $X$ whose all Hausdorff
compactifications are equivalent with the Alexandroff compactification of $X$.
Moreover, we show that Glicksberg's theorem on when the product of \v{C}%
ech-Stone compactifications of $X$ and $Y$ is the \v{C}ech-Stone
compactification of $X\times Y$ is false in some models of $\mathbf{ZF}$.

In Section 3, we show a number of necessary and sufficient conditions for $%
F\in\mathcal{E}(X)$ to have the property that the space $e_{F}X$ is compact
in \textbf{ZF}. We give a definition of a functional \v Cech-Stone
compactification and compare it in \textbf{ZF} with the standard notion of
the \v Cech-Stone compactification of a Tychonoff space. For $F\in\mathcal{E}%
(X)$, we list a number of new problems on the smallest sequentially closed
subalgebra of $C^{\ast}(X)$ which contains $F$ and all constant functions
from $C^{\ast}(X)$.

It is not obvious at all whether every Tychonoff space has a Hausdorff
compactification in $\mathbf{ZF}$. In 2016, E. Wajch asked whether there is
a model of $\mathbf{ZF}$ in which a non-compact metrizable Cantor cube can
fail to have a Hausdorff compactification (see Question 3.8 of \cite{W2}).
The following general problem is unsolved:

\begin{p}
Does there exist a model of $\mathbf{ZF}$ in which there is a Tychonoff
space which does not have any Hausdorff compactification?
\end{p}

\begin{prop} In every model of $\mathbf{ZF+UFT}$, every Tychonoff space has a Tychonoff compactification.
\end{prop}

\begin{proof} Let $X$ be a Tychonoff space in a model $\mc{M}$ of $\mathbf{ZF+UFT}$. Since, by Theorem 4.70 of \cite{Her}, all Tychonoff cubes are compact in $\mc{M}$, the space $e_F X$ for $F=C^{\ast}(X)$ is compact. Thus, in $\mc{M}$,  the pair $\langle e_F X, e_F\rangle$ is a Tychonoff compactification of $X$.
\end{proof}

Although we are still unable to solve Problem 1.1, we offer some other
results relevant to it. Of course, if $X$ is an infinite $T_1$-space, it is
easy to show a compact $T_1$-space $Y$ such that $X$ is a dense subspace of $%
Y$ and $Y\setminus X$ is a singleton. To do this, for an infinite $T_1$
space $X$, it suffices to take a point $\infty\notin X$, put $%
Y=X\cup\{\infty\}$ and define a topology in $Y$ as the collection of all
open subsets of $X$ and of all sets of the form $Y\setminus A$ where $A$ is
a finite subset of $X$.

We use the same definition of the \v Cech-Stone compactification in $\mathbf{ZF}$ as in \cite%
{HK}:

\begin{defi}
A \emph{\v{C}ech-Stone compactification} of a Hausdorff space $X$ is a Hausdorff
compactification $\langle \beta X,\beta \rangle $ of $X$ such that, for
every compact Hausdorff space $K$ and for each continuous mapping $%
f:X\rightarrow K$, there exists a continuous mapping $\tilde{f}:\beta
X\rightarrow K$ such that $f=\tilde{f}\circ \beta $.
\end{defi}

Any two \v Cech-Stone compactifications of a space $X$ are equivalent, so,
if a Hausdorff space $X$ has a \v Cech-Stone compactification, we denote any
\v Cech-Stone compactification of $X$ by $\beta X$ and call it the \v Cech-Stone compactification of $X$.

\begin{rem}
In $\mathbf{ZF}$, if a Hausdorff space $X$ has its \v Cech-Stone compactification, then, for every
Hausdorff compactification $\alpha X$ of $X$, we have $\alpha X\leq\beta X$.
\end{rem}

Basic facts about Hausdorff compactifications in $\mathbf{ZFC}$ can be found,
for instance, in \cite{Ch}, \cite{GJ}, \cite{PorW} and \cite{En}. An essential role in $\mathbf{ZFC}$-theory of Hausdorff compactifications is played by the following notion of a normal (Wallman) base which can be found, for instance, in \cite{Ch} and \cite{PorW}; however, we use its slightly modified form given in Definition 1.7 of \cite{PW}:

\begin{defi}  A family $\mc{C}$ of subsets of a topological space $X$ is called a \emph{normal} or \emph{Wallman base} for $X$ if $\mc{C}$  satisfies the following conditions:
\begin{enumerate} 
\item[(i)] $\mc{C}$ is a base for closed sets of $X$;
\item[(ii)] if $C_1, C_2\in\mc{C}$, then $C_1\cap C_2\in\mc{C}$ and $C_1\cup C_2\in\mc{C}$;
\item[(iii)] for each set $A\subseteq X$ and for each $x\in X\setminus A$, if $A$ is a singleton or $A$ is closed in $X$,  then there exists $C\in\mc{C}$ such that $x\in C\subseteq X\setminus A$;
\item[(iv)] if $A_1, A_2\in\mc{C}$ and $A_1\cap A_2=\emptyset$, then there exist $C_1, C_2\in\mc{C}$ with $A_1\cap C_1=A_2\cap C_2=\emptyset$  and $C_1\cup C_2=X$.
\end{enumerate}
\end{defi}

For Wallman-type extensions, we use the same notation as in \cite{PW}. Namely, suppose that $\mc{C}$  is a Wallman base for $X$. Let $\mc{W}(X, \mc{C})$ denote the set of all ultrafilters in $\mc{C}$. For $A\in\mc{C}$, we put
$$[A]_{\mc{C}}=\{p\in\mc{W}(X, \mc{C}): A\in p\}.$$
The  \emph{Wallman space} of $X$ corresponding to $\mc{C}$ is denoted by $\mc{W}(X, \mc{C})$ and it is the set $\mc{W}(X, \mc{C})$ equipped with the topology having the collection $\{ [A]_{\mc{C}}: A\in\mc{C}\}$ as a closed base. The canonical embedding of $X$ into $\mc{W}(X, \mc{C})$ is the mapping $h_{\mc{C}}: X\to\mc{W}(X, \mc{C})$ defined by the equality $h_{\mc{C}}(x)=\{A\in \mc{C}: x\in A\}$ for each $x\in X$. The pair $\langle\mc{W}(X, \mc{C}), h_{\mc{C}}\rangle$ is called \emph{the Wallman extension} of $X$ corresponding to $\mc{C}$ and, for simplicity, this extension is denoted also by $\mc{W}(X, \mc{C})$. In the case when $X$ is a discrete space, the Wallman space $\mc{W}(X, \mc{P}(X))$ corresponding to the power set $\mc{P}(X)$ is usually called the \emph{Stone space} of $X$ and it is denoted by $\mc{S}(X)$.

 A topological space $X$ is called \emph{semi-normal} if $X$ has a Wallman base (cf. \cite{F} and \cite{PW}). In the light of Theorem 2.8 of \cite{PW}, $\mathbf{UFT}$ is an equivalent of the following sentence: For every semi-normal space $X$ and for every normal base $\mc{C}$ of $X$, the Wallman space $\mc{W}(X, \mc{C})$ is compact. A Hausdorff compactification $\alpha X$ of a space $X$ is called a \emph{Wallman-type compactification} of $X$ if there exists a normal base $\mc{C}$ of $X$ such that the space $\mc{W}(X, \mc{C})$ is compact and the compactification $\mc{W}(X, \mc{C})$ of $X$ is equivalent with $\alpha X$. It was proved in \cite{U} that not all Hausdorff compactifications in $\mathbf{ZFC}$ are of Wallman type; however, a satisfactory solution to the following problem is unknown:
 
 \begin{p} Can it be proved in $\mathbf{ZF}$ that there are discrete spaces that have Hausdorff compactifications which are not of Wallman type?
 \end{p}
 
\begin{rem} It is worth to notice that \v Sapiro's result given in \cite{\v S}  that all Hausdorff compactifications are of Wallman type if and only if all Hausdorff compactifications of discrete spaces are of Wallman type is provable in $\mathbf{ZF}$. Historical remarks about the results of \cite{\v S} are given in \cite{ChF}.
 \end{rem}

For a topological space $X$, let $\mc{Z}(X)$ stand for the collection of all zero-sets in $X$. Then $\mc{Z}(X)=\{ f^{-1}(0); f\in C^{\ast}(X)\}$. Of course, $\mc{Z}(X)$ is a normal base for $X$ if and only if $X$ is a Tychonoff space.  It is well known that, in $\mathbf{ZFC}$,  if $X$ is a Tychonoff space, then its \v{C}ech-Stone compactification is of Wallman type because it is equivalent with the Wallman extension of $X$ corresponding to the normal base $\mc{Z}(X)$ (cf., e.g., \cite{Ch}, \cite{F},  \cite{GJ}, \cite{PorW} and \cite{PW}). That the Wallman space $\mc{W}(X, \mc{Z}(X))$ is compact for every Tychonoff space $X$ is an equivalent of $\mathbf{UFT}$ (see, e.g., Theorem 2.8 of \cite{PW}). Some relatively new results on Hausdorff compactifications in $\mathbf{ZF}$ have been obtained in \cite{HerK}, \cite{HK}, \cite{HKT},  \cite{K2}, \cite{KT} and, for Delfs-Knebusch generalized topological spaces (applicable to topological spaces),  in \cite{PW}. In \cite{AL}, there is a well-written chapter on a history of Hausdorff compactifications in $\mathbf{ZFC}$ (see \cite{ChF}).  Unfortunately, not much is known about Hausdorff compactifications in $\mathbf{ZF}$. In this article, to start a systematic study of Hausdorff compactifications in $\mathbf{ZF}$,  we put in order basic notions concerning them, as well as we  show newly discovered important differences between $\mathbf{ZF}$-theory and $\mathbf{ZFC}$-theory of Hausdorff compactifications. 

\section{Strange Hausdorff compactifications}

\begin{defi}
Let $\alpha X$ be a Hausdorff compactification of a Tychonoff space $X$. We
say that:

\begin{enumerate}
\item[(i)] $\alpha X$ is \emph{generated by a set of functions} if there
exists $F\in\mathcal{E}(X)$ such that the space $e_F X$ is compact and the
compactifications $\alpha X$ and $e_F X$ of $X$ are equivalent;

\item[(ii)] $\alpha X$ is \emph{generated by a set} $F\in\mathcal{E}(X)$ if the
space $e_F X$ is compact, while the compactifications $e_F X$ and $\alpha X$
of $X$ are equivalent;

\item[(iii)] $\alpha X$ is \emph{strange} if it is not generated by a set of
functions.
\end{enumerate}
\end{defi}

\begin{defi} For a topological space $X$, we say that:
\begin{enumerate}
\item[(i)]  $\mathbf{UL}(X)$ holds or, equivalently, $X$ \emph{satisfies Urysohn's Lemma} or, equivalently, $X$ is a U space if, for each pair of disjoint closed subsets $A, B$ of $X$ there exists $f\in C^{\ast}(X)$ such that $A\subseteq f^{-1}(0)$ and $B\subseteq f^{-1}(1)$;
\item[(ii)] $\mathbf{TET}(X)$ holds or, equivalently, $X$ \emph{satisfies Tietze's Extension Theorem} or, equivalently, $X$ is a T space if, for each closed subspace $P$ of $X$, every function $f\in C(P)$ is extendable to a function from $C(X)$.
\end{enumerate}
\end{defi}

Let us denote Form 78 of \cite{HR} (Urysohn's Lemma) by $\mathbf{UL}$. Then $\mathbf{UL}$ is the sentence: $\mathbf{UL}(X)$ holds for every normal space $X$.  We denote Form 375 of \cite{HR} (Tietze-Urysohn Extension Theorem) by $\mathbf{TET}$. Then  $\mathbf{TET}$ is the sentence: $\mathbf{TET}(X)$ holds for every normal space $X$. The principle of dependent choices (Form 43 in \cite{HR}) is denoted by $\mathbf{DC}$.
 
\begin{rem}  In \cite{HKRR}, spaces that satisfy Urysohn's Lemma were called $\text{U}$ spaces, while spaces that satisfy Tietze's Extension Theorem were called $\text{T}$ spaces,  $\mathbf{UL}$ was denoted by $\mathbf{NU}$, while $\mathbf{TET}$ by $\mathbf{NT}$. Of course, every T space is a U space and all U spaces are normal. In \cite{HKRR},  $\mathbf{NU}$ ($\mathbf{NT}$, resp.) is an abbreviation to: \emph{"Every normal space is a} U \emph{space."} (\emph{"Every normal space is a} T \emph{space."}, resp.). However, in this article, we find it  more natural to denote Urysohn's Lemma by $\mathbf{UL}$ and Tietze's Extension Theorem by $\mathbf{TET}$. It is well known that $\mathbf{UL}$ and $\mathbf{TET}$ are independent of $\mathbf{ZF}$.  Of course, it is true in $\mathbf{ZF}$ that if $X$ is a topological space such that $\mathbf{TET}(X)$ holds, then $\mathbf{UL}(X)$ also holds. It is known that $\mathbf{ZF+DC}$ implies both $\mathbf{UL}$ and $\mathbf{TET}$ (cf. entries (43, 78) and (43, 375) on pages 339 and 386 in \cite{HR}). Hence, in $\mathbf{ZF+DC}$, a topological space $X$ satisfies $\mathbf{UL}(X)$ if and only if $\mathbf{TET}(X)$ holds. In \cite{HKRR}, it was shown that there is a model $\mathcal{M}$ of $\mathbf{ZF}$ in which there is a compact Hausdorff space $X$ such that $\mathbf{UL}(X)$ holds and $\mathbf{TET}(X)$ fails in $\mathcal{M}$.  However, it is an open question, already posed in \cite{HKRR},  whether $\mathbf{UL}$ implies $\mathbf{TET}$. 
\end{rem}

We can easily obtain the following results:

\begin{prop} $\mathbf{[ZF+UL]}$ An arbitrary Hausdorff compactification is not strange.
\end{prop}

\begin{prop}
$\mathbf{[ZF]}$ A Hausdorff compactification $\alpha X$ of a non-empty 
Tychonoff space $X$ is not strange if and only if $\alpha X$ is completely
regular.
\end{prop}

\begin{prop}
$\mathbf{[ZF]}$ If a Hausdorff compactification $\gamma X$ of a non-empty 
topological space $X$ is completely regular, then $\gamma X$ is generated by 
$C_{\gamma }(X)$.
\end{prop}

\begin{defi}
Let $n$ be a positive integer. It is said that a compactification $\alpha X$
of a topological $X$ is an $n$-point compactification of $X$ if the
remainder $\alpha X\setminus X$ consists of exactly $n$ points.
\end{defi}

The Alexandroff compactification of a locally compact, non-compact Hausdorff
space $X$ is every Hausdorff compactification of $X$ with a one-point
remainder. However, the following notion of the Alexandroff compactification of a
topological space is also useful:

\begin{defi}
Let $\langle X, \tau_X\rangle$ be a non-compact topological space and let $%
\infty$ be an element which does not belong to $X$. Denote by $\mathcal{K}_X$
the collection of all simultaneously closed and compact sets of $\langle X,
\tau_X\rangle$. We put $\alpha_{a}X= X\cup\{ \infty\}$ and $%
\tau=\tau_X\cup\{ U\subseteq \alpha_{a}X: \alpha_a X\setminus U\in\mathcal{K}%
_X\}$. Then the topological space $\langle \alpha_a X, \tau\rangle$ is
called the \emph{Alexandroff compactification} of $\langle X, \tau_X\rangle$ and we denote it by $%
\alpha_a X$.
\end{defi}

We are going to give partial solutions to the following open problem:

\begin{p}
Is there a model of $\mathbf{ZF}$ in which there exists a Hausdorff, not
completely regular compactification $\gamma X$ of a Tychonoff space $X$ such
that $\gamma X\setminus X$ is completely regular?
\end{p}

For a space $X$, let $Coz(X)=\{ X\setminus A: A\in\mc{Z}(X)\}$. Members of $Coz(X)$ are called cozero-sets of $X$. Basic properties of zero-sets and cozero-sets can be found in \cite{GJ} and \cite{En}.

We are going to prove in $\mathbf{ZF}$ that all Hausdorff compactifications with finite remainders of Tychonoff spaces are not strange. To do this well, we need a proof in $\mathbf{ZF}$ of Theorem 3.1.7 of \cite{En}; however, since the axiom of choice is involved in the proof to Theorem 3.1.7 in \cite{En}, let us state the following lemma and give its subtle proof in $\mathbf{ZF}$: 

\begin{lem} $\mathbf{[ZF]}$ Let $K$ be a  compact subset of a completely regular space $X$ and let $A$ be a closed subset of $X$ such that $K\cap A=\emptyset$.  Then there exists a function $f\in C^{\ast}(X)$ such that $A\subseteq f^{-1}(0)$ and $K\subseteq f^{-1}(1)$.
\end{lem} 

\begin{proof} Let $\mc{V}=\{ V\in Coz(X): \text{cl}_X V\subseteq X\setminus A\}$. Since $X$ is completely regular, we have $K\subseteq\bigcup\mc{V}$. By the compactness of $K$, there exists a finite collection $\mc{U}\subseteq\mc{V}$ such that $K\subseteq\bigcup\mc{U}$. A  finite union of cozero-sets is a cozero-set; thus,  the set $U_0=\bigcup\mc{U}$ is a cozero-set of $X$. There exists a continuous function $g: X\to [0, 1]$ such that $U_0=g^{-1}((0, 1])$. Then $g^{-1}(0)\cap K=\emptyset$. It follows from the continuity of $g$ and from the compactness of $K$ that there exists a positive integer $n_0$  such that $g^{-1}([0, \frac{1}{n_0}])\cap K=\emptyset$. The sets $C=g^{-1}(\{0\})$ and $D=g^{-1}([\frac{1}{n_0}, 1])$ are disjoint zero-sets in $X$, $K\subseteq D$ and $A\subseteq C$. Since disjoint zero-sets are functionally separated (cf.  1.10 in \cite{GJ} or Theorem 1.5.14 in \cite{En}), there exists a continuous function $f:X\to [0, 1]$ such that $C\subseteq f^{-1}(0)$ and $D\subseteq f^{-1}(1)$. Then $A\subseteq f^{-1}(0)$ and $K\subseteq f^{-1}(1)$.
\end{proof}

\begin{cor} $\mathbf{[ZF]}$ Every compact completely regular space $X$ satisfies $\mathbf{UL}(X)$. 
\end{cor}

\begin{prop}
$\mathbf{[ZF]}$ If $X$ is a Tychonoff, locally compact non-compact space,
then the one-point Hausdorff compactification of $X$ is not strange.
\end{prop}

\begin{proof} 
Let us fix a closed subset $A$ of $\alpha _{a}X$ and suppose that $x\in \alpha _{a}X\setminus
A$. We consider the following cases:\smallskip 

(i) $\infty \in A$. In this case $x\in X$. Since every Hausdorff compact space is normal, there exists a pair $U, V$ of disjoint open sets in $\alpha_a X$ such that $x\in U$ and $A\subseteq V$. Then $x\notin\text{cl}_{\alpha_a X}V$. Since $X$ is Tychonoff, there exists a function $f\in C^{\ast}(X)$ such that $f(x)=0$ and $X\cap\text{cl}_{\alpha_a X} V\subseteq f^{-1}(1)$. We define a function $F: \alpha_a X\to\mathbb{R}$  by putting $F(t)=f(t)$ for each $t\in X$ and $F(\infty)=1$. To check that $F$ is continuous, suppose that $D$ is a closed subset of $\mathbb{R}$. Then $F^{-1}(D)=f^{-1}(D)\subseteq \alpha_a X\setminus V$ when $1\notin D$, while $F^{-1}(D)=[f^{-1}(D)\cap(\alpha_a X\setminus V)]\cup\text{cl}_{\alpha_a X}V$ when $1\in D$. This, together with the continuity of $f$,  implies that  $F^{-1}(D)$ is closed in $\alpha_a X$, so $F\in C^{\ast}(\alpha_a X)$. Of course, $A\subseteq F^{-1}(1)$ and $F(x)=0$.

(ii) $\infty \notin A$. In this case $A$ is a compact subset of $X$.  Working similarly to case (i), we can find a pair $U, V$ of disjoint open subsets of $\alpha_a X$ such that $A\subseteq U$ and $\{\infty, x\}\subseteq V$. Since $X$ is completely regular, it follows from Lemma 2.10 that there exists a function $f\in C^{\ast}(X)$ such that $A\subseteq f^{-1}(0)$ and $X\cap\text{cl}_{\alpha_a X}V\subseteq f^{-1}(1)$. We define $F\in C^{\ast}(\alpha_a X)$ by putting $F(t)=f(t)$ for each $t\in X$ and $F(\infty)=1$. Then $A\subseteq F^{-1}(0)$ and $F(x)=1$.
\end{proof}

Proposition 2.12 can be generalized to the following:

\begin{prop}
$\mathbf{[ZF]}$ Every Hausdorff
compactification $\alpha X$  a non-compact locally compact Tychonoff $X$ with a finite remainder $\alpha X\setminus X$ is completely regular.
\end{prop}

\begin{proof}  Let $\alpha X$ be a Hausdorff compactification of $X$ such that $\alpha X\setminus X$ is finite. For $n\in\omega$, suppose that  $\alpha X\setminus X$ is of cardinality $n$.  Let $\alpha X\setminus X=\{y_i: i\in n\}$. Let $A$ be a closed subset of $\alpha X$ and let $y\in\alpha X\setminus A$. If $y\in X$, we know from Proposition 2.12 that there exists a function $h\in C_{\alpha_a}(X)$ such that $h(y)=0$ and $A\cap X\subseteq h^{-1}(1)$. Since $\alpha_a X\leq \alpha X$, the function $h$ is continously extendable over $\alpha X$. If $\tilde{h}$ is the continuous extension of $h$ over $\alpha X$, then $\tilde{h}(y)=0$ and $A\subseteq \tilde{h}^{-1}(1)$.

Now, consider the case when $y\in\alpha X\setminus X$.  There is a collection $\{V_i: i\in n\}$ of pairwise disjoint open sets in $\alpha X$ such that $y_i\in V_i$ for each $i\in n$. Let $K=\alpha X\setminus \bigcup_{i\in n}V_i$ and let $A_i=K\cup (A\cap V_i)$ for each $i\in n$. Consider any $i\in n$. Notice that $A_i=K\cup [A\cap(\alpha X\setminus\bigcup\{V_j: j\in n\setminus\{i\}\})]$, so $A_i$ is closed in $K\cup V_i$. Of course, $K\cup V_i=\alpha X\setminus\bigcup\{V_j: j\in n\setminus\{i\}\}$ is a one-point compactification of $K\cup (X\cap V_i)$. Therefore,  if  $y_i\notin A$, it follows from Proposition 2.12  that there exists a continuous function $f_i: K\cup V_i\to [0, 1]$ such that $f_i(y_i)=0$ and $A_i\subseteq f_i^{-1}(1)$. If $i\in n$ is such that $y_i\in A$, we put $f_i(z)=1$ for each $z\in K\cup V_i$. We define a function $f:\alpha X\to [0, 1]$ as follows: if $i\in n$ and $z\in K\cup V_i$, then $f(z)=f_i(z)$. Clearly, $A\subseteq f^{-1}(1)$ and $f(y_i)=0$ for each $i\in n$ such that $y_i\notin A$. Let us prove that $f$ is continuous. To this aim, consider any closed in $\mathbb{R}$ set $D$.  If $i\in n$, the set $f_i^{-1}(D)$  is closed in $K\cup V_i$. Since $K\cup V_i$ is closed in $\alpha X$ for each $i\in n$, we have that $f^{-1}(D)$ is closed in $\alpha X$ because $f^{-1}(D)=\bigcup\{ f_i^{-1}(D): i\in n\}$. Finally, to show that $f(y)=0$ and $A\subseteq f^{-1}(1)$, it suffices to notice that since $y\in\alpha X\setminus X$, there exists $i\in n$ such that $y=y_i$ and, of course, $y_i\notin A$. 
\end{proof}

\begin{prop} $\mathbf{[ZF]}$ Suppose that $\alpha X$ is a Hausdorff compactification of a Tychonoff space $X$ such that $\alpha X\setminus X$ is homeomorphic with the Alexandroff compactification of the discrete space $\omega$. Then $\alpha X$ is completely regular.
\end{prop}

\begin{proof} We may assume that $\alpha X\setminus X=\omega+1$ where $\omega+1$ is equipped with the usual order topology on ordinal numbers. Notice that $X$ is locally compact because $X$ is open in $\alpha X$. Therefore, $\alpha_a X$ is a Hausdorff compactification of $X$.  Let $A$ be a closed subset of $\alpha X$ and let $y\in\alpha X\setminus A$. If $y\in X$, we know from Proposition 2.12 that there exists $f\in C(\alpha_a X)$ such that $f(y)=0$ and $(A\cap X)\cup(\alpha_a X\setminus X)\subseteq f^{-1}(1)$. If $\tilde f$ is the continuous extension of $f\vert_X$ over $\alpha X$, then $\tilde{f}(y)=0$ and $A\subseteq\tilde{f}^{-1}(1)$. Now, assume that $y\in\alpha X\setminus X$ and consider the following cases (i) and (ii):\smallskip

(i)  $y=\omega$. In this case,  there exists an open in $\alpha X$ set $V$ such that $A\cap\text{cl}_{\alpha X}V=\emptyset$, $y\in V$ and the set $(\alpha X\setminus X)\setminus V$ is finite. Let $W$ be an open set in $\alpha X$ such that $A\subseteq W$ and $(\text{cl}_{\alpha X} W)\cap(\text{cl}_{\alpha X} V)=\emptyset$. Put $Y=X\cap\text{cl}_{\alpha X} W$ and $\gamma Y=\text{cl}_{\alpha X}Y=\text{cl}_{\alpha X} W$. Then $\gamma Y$ is a Hausdorff compactification of the Tychonoff space $Y$ such that $\gamma Y\setminus Y$ is finite.  In  view of Proposition 2.13, the space $\gamma Y$ is completely regular. The sets $\text{bd}_{\alpha X} W$ and $A$ are disjoint and both compact in  the compact Tychonoff space $\gamma Y$. Thus, it follows from Lemma 2.10 that there exists a continuous function $g:\gamma Y\to [0, 1]$ such that $\text{bd}_{\alpha X} W\subseteq g^{-1}(0)$ and $A\subseteq g^{-1}(1)$. We define a function $\tilde{g}:\alpha X\to [0, 1]$ putting $\tilde{g}(z)=0$  for each $z\in\alpha X\setminus\text{cl}_{\alpha X} W$, while $\tilde{g}(z)=g(z)$ for each $z\in\gamma Y$. The function $\tilde{g}$ is continuous on $\alpha X$; moreover,  $\tilde{g}(y)=0$ and $A\subseteq \tilde{g}^{-1}(1)$. 

(ii) $y\in \omega$. Then there exists an open neigbourhood $W(y)$ of $y$ in $\alpha X$ such that $A\cap\text{cl}_{\alpha X} W(y)=\emptyset$ and $[\text{cl}_{\alpha X} W(y)]\cap(\alpha X\setminus X)=\{y\}$. Since $\text{cl}_{\alpha X} W(y)$ is a one-point Hausdorff compactification of $X\cap\text{cl}_{\alpha X}W(y)$, it follows from Proposition 2.12 that there exists a continuos function $h:\text{cl}_{\alpha X}W(y)\to [0, 1]$ such that $h(y)=0$ and $\text{bd}_{\alpha X} W(y)\subseteq h^{-1}(1)$. We define a function $\tilde{h}:\alpha X\to [0, 1]$ as follows: if $z\in \text{cl}_{\alpha X}W(y)$, then $\tilde{h}(z)=h(z)$; if $z\in\alpha X\setminus W(y)$, then $\tilde{h}(z)=1$. The function $\tilde{h}$ is continuous, $\tilde{h}(y)=0$ and $A\subseteq \tilde{h}^{-1}(1)$.  
\end{proof}

We recall that a subspace $P$ of a space $X$ is called $C^{\ast}$-embedded in $X$ if each function from $C^{\ast}(P)$ is continuously extendable over $X$ (cf. 1.13 in \cite{GJ} or Definition 1.31 in \cite{Ch}).

\begin{thm} $\mathbf{[ZF]}$ Let $\alpha X$ be a Hausdorff compactification of a locally compact Tychonoff space $X$ such that $\alpha X\setminus X$ is completely regular and $C^{\ast}$-embedded in $\alpha X$. Then $\alpha X$ is completely regular.
\end{thm}

\begin{proof} Let $A$ be a closed subset of $\alpha X$ and let $y\in\alpha X\setminus A$. If $y\in X$, we know from Proposition 2.12 that there exists a function $\psi\in C(\alpha_a X)$ such that $\psi(y)=1$ and $(A\cap X)\cup(\alpha_a X\setminus X)\subseteq \psi^{-1}(0)$. If $f$ is the restriction of $\psi$ to $X$, while $\tilde{f}$ is the continuous extension of $f$ over $\alpha X$, then $\tilde{f}(y)=1$ and $A\subseteq \tilde{f}^{-1}(0)$. 

Now, consider the case when $y\in\alpha X\setminus X$. Since $\alpha X\setminus X$ is completely regular, there exists a continuous function $g:\alpha X\setminus X\to [0, 1]$ such that $g(y)=1$ and $A\cap(\alpha X\setminus X)\subseteq g^{-1}(0)$. Since $\alpha X\setminus X$ is $C^{\ast}$-embedded in $\alpha X$, the function $g$ has a continuous extension $\tilde{g}:\alpha X\to [0, 1]$.  Let $B= A\cap\tilde{g}^{-1}([\frac{1}{2}, 1])$. Then $B$ is a compact subset of $X$. In the light of Proposition 2.12, the Alexandroff compactification $\alpha_a X$ is completely regular, so there exists  a continuous function $\phi: \alpha_a X\to [0, 1]$ such that $\phi(\alpha_a X\setminus X)=\{1\}$ and $B\subseteq \phi^{-1}(0)$. Since $\alpha_a X\leq\alpha X$, the function $\phi\vert_X$ has a continuous extension $\tilde{h}:\alpha X\to [0, 1]$. Let us consider the function $\kappa:\alpha X\to [0, 1]$  defined by $\kappa=\min\{\tilde{g}, \tilde{h}\}$. It is clear that $\kappa(y)=1$. If $z\in B$, then $\kappa(z)=0$ because $\tilde{h}(z)=0$. If $z\in A\setminus B$, then $\tilde{g}(z)\leq \frac{1}{2}$, so $\kappa(z)\leq\frac{1}{2}$. This implies that $A\subseteq \kappa^{-1}([0, \frac{1}{2}])$. Let $C=\kappa^{-1}(1)$ and $D=\kappa^{-1}([0, \frac{1}{2}])$. Then $C, D$  are disjoint zero-sets in $\alpha X$ such that $y\in C$ and $A\subseteq D$. This proves that $\alpha X$ is completely regular because disjoint zero-sets are functionally separated.
\end{proof}

\begin{rem}  In view of Remark 2.3 and Lemma 2.10, it holds true in every model of
\textbf{ZF }$+$ \textbf{DC} that if $X$ is a compact Tychonoff space, then  $\mathbf{TET}(X)$ is satisfied. It was shown in Section 3 of  \cite{HKRR} that  there is a model of $\mathbf{ZF}$ in which a compact Tychonoff space $X$ need not satisfy $\mathbf{TET}(X)$. 
\end{rem}

\begin{prop} $\mathbf{[ZF]}$ Suppose that $\alpha X$ is a Hausdorff compactification of a Tychonoff space $X$ such that $\alpha X\setminus X$ is finite. Then $\alpha X\setminus X$ is $C^{\ast}$-embedded in $\alpha X$. 
\end{prop}

\begin{proof} We may assume that $X$ is non-compact. Let $n\in\omega$ be equipotent with $\alpha X\setminus X$. We fix a function $f: \alpha X\setminus X\to\mathbb{R}$ and put $D=f(\alpha X\setminus X)$. Assume that $\alpha X\setminus X=\{ y_i: i\in n\}$. There is a collection $\{V_i: i\in n\}$ of pairwise disjoint open sets in $\alpha X$ such that $y_i\in V_i$ for each $i\in n$. For $d\in D$, let $N(d)=\{ i\in n: f(y_i)=d\}$ and $A(d)=\alpha X\setminus\bigcup\{V_i: i\in N(d)\}$. In the light of Proposition 2.13, the space $\alpha X$ is Tychonoff. Thus, for each $d\in D$, there exists a continuous function $g_d:\alpha X\to[d, 1+ \max D]$ such that $f^{-1}(d)\subseteq g_d^{-1}(d)$ and $A(d)\subseteq g_d^{-1}(1+\max D)$. Let $g(t)=\min\{g_d(t): d\in D\}$ for each $t\in\alpha X$. Then $g\in C(\alpha X)$ and $g(t)=f(t)$ for each $t\in\alpha X\setminus X$. 
\end{proof}

Of course, one can also deduce from Proposition 2.13 and the $\mathbf{ZFC}$-proof to Tietze-Urysohn extension theorem that Proposition 2.17 holds in $\mathbf{ZF}$; however, we prefer a simpler, direct $\mathbf{ZF}$-proof to it. Among other facts, we are going to show that it may happen in a model of $\mathbf{ZF}$ that, for a Hausdorff compactification $\alpha X$ of a locally compact Tychonoff space such that $\alpha X\setminus X$ is homeomorphic with $\omega+1$, the remainder $\alpha X\setminus X$ can fail to be $C^{\ast}$-embedded in $\alpha X$. It might be interesting to know the place of the following sentences $C^{\ast}\mathbf{R}$ and $C^{\ast}\mathbf{R}[\omega]$ in the hierarchy of choice principles:

$C^{\ast}\mathbf{R}$: For every locally compact Tychonoff space $X$ and for every Tychonoff compactification $\alpha X$ of $X$, the remainder $\alpha X\setminus X$ is $C^{\ast}$-embedded in $\alpha X$.

$C^{\ast}\mathbf{R}[\omega]$: For every locally compact  Tychonoff space $X$ and for every Hausdorff compactification $\alpha X$ of $X$ such that $\alpha X\setminus X$ is homeomorphic with $\omega+1$, the remainder $\alpha X\setminus X$ is $C^{\ast}$-embedded in $\alpha X$.

As usual, we denote by $\mathbf{CMC}$ the axiom of countable multiple choice which states that for each sequence $(X_n)_{n\in\omega}$ of non-empty sets there exists a sequence $(F_n)_{n\in\omega}$ of non-empty finite subsets $F_n$ of $X_n$ (see Form 126 in \cite{HR} and Definition 2.10 in \cite{Her}). The axiom of countable choice (Form 8 in \cite{HR}), denoted by $\mathbf{CC}$ in Definition 2.5 of \cite{Her} and by $\mathbf{CAC}$ in many articles (see, for instance, \cite{HerK} and \cite{HKRR}), states that every non-empty countable collection of non-empty sets has a choice function. Let us recall the following van Douwen's choice principle (Form 119 in \cite{HR}, as well as $\mathbf{CC}(\mathbb{Z})$ on page 79 in \cite{Her}) which was introduced in \cite{vD} and denoted by $\mathbf{vDPC}(\omega)$ in \cite{HKRR}:

$\mathbf{vDCP}(\omega)$: For every family $\{ \langle A_i, \leq_i\rangle: i\in\omega\}$ such that each $\langle A_i, \leq_i\rangle$ is a linearly ordered set isomorphic with the set $\mathbb{Z}$ of integers equipped with the standard order, the family $\{ A_i; i\in\omega\}$ has a choice function.

It was shown in \cite{HKRR} that $\mathbf{vDCP}(\omega)$ is strictly weaker than $\mathbf{CMC}$. For significant applications of models in which $\mathbf{vDCP}(\omega)$ fails, the following construction was used, for instance, in \cite{vD}, \cite{HKRR}, \cite{W2} and in Section 4.7 of \cite{Her}:

Let $\mc{A}=\{ \langle A_i, \leq_i\rangle: i\in\omega\}$ be a collection of linearly ordered sets $\langle A_i, \leq_i\rangle$ isomorphic with the set $\mathbb{Z}$ of integers equipped with the standard order. Let $A=\bigcup\{A_i: i\in\omega\}$. Fix sets $\bar{A}=\{a_i: i\in\omega\}$ and $\bar{B}=\{b_i: i\in\omega\}$ of pairwise distinct  elements such that $A\cap(\bar{A}\cup\bar{B})=\emptyset=\bar{A}\cap\bar{B}$. Put $X_i=A_i\cup\{ a_i, b_i\}$ and extend the order $\leq_i$ to a linear order $\leq_i$ on $X_i$  by requiring that $a_i$ is the smallest element of $\langle X_i, \leq_i\rangle$, while $b_i$ is the largest element of $\langle X_i, \leq_i\rangle$. For simplicity, without any loss of generality, we may assume that $X_i\cap X_j=\emptyset$ for each pair $i,j$  of distinct elements of $\omega$. We equip each $X_i$ with the order topology induced by $\leq_i$. We denote by $X[\mc{A}]$ the disjoint union (the sum) of the linearly ordered topological spaces $X_i$ where $i\in\omega$. The spaces $X_i$ are all metrizable, so Tychonoff. Clearly, the space $X[\mc{A}]$ is locally compact.  It is easy to prove in $\mathbf{ZF}$ that sums of completely regular spaces are completely regular. Hence  $X[\mc{A}]$ is also completely regular in $\mathbf{ZF}$. 

\begin{thm} The following implications are true in every model of $\mathbf{ZF}$:
\begin{enumerate}
\item[(i)] $C^{\ast}\mathbf{R}[\omega]$ implies $\mathbf{vDCP}(\omega)$;
\item[(ii)] the conjuction of $\mathbf{UFT}$ and $C^{\ast}\mathbf{R}[\omega]$ implies $\mathbf{CMC}$;
\item[(iii)] $\mathbf{CC}$ implies $C^{\ast}\mathbf{R}[\omega]$.
\end{enumerate}
\end{thm}

\begin{proof} (i) Let us fix a family $\mc{A}=\{ \langle A_i, \leq_i\rangle: i\in\omega\}$ as in the definition of $\mathbf{vDCP}(\omega)$, as well as sets $\bar{A}=\{a_i: i\in\omega\}, \bar{B}=\{b_i: i\in \omega\}$ as in the construction of $X[\mc{A}]$ described above. We consider the Alexandroff compactification $\alpha_a X[\mc{A}]=X[\mc{A}]\cup\{\infty\}$ of $X[\mc{A}]$, the set $K=\bar{A}\cup\{\infty\}\cup\bar{B}$ and the subspace $X=(\alpha_a X[\mc{A}])\setminus K$ of $\alpha_a X[\mc{A}]$.  Of course, the space $X$ is discrete and dense in $\alpha_a X[\mc{A}]$. In view of Proposition 2.12, the compactification $\alpha X=\alpha_a X[\mc{A}]$ of $X$ is Tychonoff.   We define a continuous function $g: K\to [-1, 1]$ by putting $g(\infty)=0$, while $g(a_i)=\frac{-1}{i+1}$ and $g(b_i)=\frac{1}{i+1}$ for each $i\in\omega$. Suppose that $K$ is $C^{\ast}$-embedded in $\alpha X$. Then $g$ has a continuous extension to a function $\tilde{g}:\alpha X\to [-1, 1]$. In much the same way,  as in Section 3 of \cite{HKRR}, for each $i\in\omega$, we can define $t(i)=\max\{ x\in A_i: \tilde{g}(x)<0\}$ to obtain a choice function $t\in\prod_{i\in\omega} A_i$. This, together with the fact that $K$ is homeomorphic with the Alexandroff compactification $\omega+1$ of $\omega$, implies that (i) holds.

(ii)  Now, let us suppose that $\mathbf{CMC}$ is false.  In this case, notice that it was shown in the proof to Theorem 3 of \cite{HKRR} that there exists a Tychonoff space $Z$ such that, for a compact subset $C$ of $Z$, the set $Z\setminus C$ is dense in $Z$, while $C$ is not $C^{\ast}$-embedded in $Z$ and $C$  is homeomophic with $\omega+1$.  Assume that $\mathbf{UFT}$ is satisfied.  It  follows from Proposition 1.2 that the space $Z$ has a Tychonoff compactification $\gamma Z$. The subspace $Z_0=(\gamma Z)\setminus C$ of $\gamma Z$ is a locally compact space such that the remainder $C=(\gamma Z)\setminus Z_0$ is not $C^{\ast}$-embedded in the Tychonoff compactification $\gamma Z$ of $Z_0$. 

(iii)  Assume that $\mathbf{CC}$ holds. Let $\alpha X$ be a Hausdorff compactification of a Tychonoff space $X$ such that $\alpha X\setminus X=\omega+1$. Denote by $\mc{F}$ the collection of all non-empty finite subsets of $\omega$. Then $\mc{F}$ is countable. Let $\mc{F}=\{F_n: n\in\omega\}$. For each $n\in\omega$, let $\mc{G}_n$ be the collection of all functions $g\in C(\alpha X)$ such that $(\omega+1)\setminus F_n\subseteq g^{-1}(0)$ and $F_n\subseteq g^{-1}(1)$. The collections $\mc{G}_n$ are all non-empty because, by Proposition 2.14, the space $\alpha X$ is Tychonoff. Since $\mathbf{CC}$ holds, the collection $\{\mc{G}_n: n\in\omega\}$ has a choice function.  Let $G\in\prod_{n\in\omega}\mc{G}_n$. Let $A,B$ be a pair of non-empty disjoint closed subsets of $\omega+1$. Then $A\in\mc{F}$ or $B\in\mc{F}$. Suppose that $n_0\in\omega$ is such that $A=F_{n_0}$. The function $g=G(n_0)$ is such that $B\subseteq g^{-1}(0)$ and $A\subseteq g^{-1}(1)$. With this observation in hand, if $f\in C(\omega+1)$, we can slightly modify the well-known standard $\mathbf{ZFC}$-proof of Tietze-Urysohn Extension Theorem to find in $\mathbf{ZF+CC}$ a continuous extension of $f$ over $\alpha X$. 
\end{proof}

\begin{cor} The following sentences are relatively consistent with $\mathbf{ZF}$:
\begin{enumerate}
\item[(i)] There exists a Tychonoff compactification $\gamma Y$ of a locally compact space $Y$ such that $\mathbf{TET}(\gamma Y)$ fails but $\gamma Y\setminus Y$ is $C^{\ast}$-embedded in $\gamma Y$. 
\item[(ii)] $C^{\ast}\mathbf{R}[\omega]$ is false.
\end{enumerate}
\end{cor}
\begin{proof} Let us consider the space $X=\alpha_a X[\mc{A}]\setminus K$ and its compactification $\alpha X=\alpha_a X[\mc{A}]$ used in the proof to (i) of Theorem 2.18. Let $c\in\alpha X$ be an accumulation point of $X$, let $Y=\alpha X\setminus\{c\}$ and $\gamma Y=\alpha X$. It was shown in the proof to Theorem 2.18 (i) that if $\mc{A}$ is such that $\{A_i: i\in\omega\}$ does not have a choice function, then $\alpha X$ is not a T space. To complete the proof, it suffices to use Theorem 2.18 together with the fact that there is a model of $\mathbf{ZF}$ in which $\mathbf{vDCP}$ fails.
\end{proof}

We are going to prove that it is relatively consistent with \textbf{ZF} that
there exists a Tychonoff space which has a strange Hausdorff
compactification. We shall deduce several surprising consequences of the
existence of strange compactifications in some models of \textbf{ZF}.

\begin{thm}
$\mathbf{[ZF]}$ Let $Y$ be a given non-empty  compact Hausdorff space. Then
there exist a discrete space $D_{Y}$ and a Hausdorff compactification $%
\gamma D_{Y}$ of $D_{Y}$ such that $(\gamma D_{Y})\setminus D_{Y}$ is
homeomorphic with $Y$.
\end{thm}

\begin{proof}  Let $D_Y=(\omega\times Y)\setminus(\{0\}\times Y)$ be considered with its discrete topology. Let $\gamma D_Y=\omega\times Y$ be equipped with the following topology:
\begin{itemize}
\item[(i)] all points of $(\gamma D_Y)\setminus (\{0\}\times Y)$ are isolated;
\item[(ii)] if $y\in Y$, then a base of neighbourhoods of the point $\langle 0, y\rangle$ in $\gamma D_Y$ consists of all sets of the form: $(\omega\times U)\setminus K$ where $K$ is a finite subset of $D_Y$, while $U$ is open in $Y$ and $y\in U$.
\end{itemize}
Obviously,  the topological space $\gamma D_Y$ is a compact Hausdorff space such that $D_Y$ is a dense subspace of $\gamma D_Y$, while the space $Y$ is homeomorphic with the remainder $(\gamma D_Y)\setminus D_Y$.
\end{proof}

\begin{thm}
The following sentences are relatively consistent with $\mathbf{ZF}$:

\begin{enumerate}
\item[(i)] There exists a discrete space which has a strange
compactification.

\item[(ii)] There exists a Tychonoff space $X$ which has non-equivalent
Hausdorff compactifications $\alpha X$ and $\gamma X$ such that $%
C_{\alpha}(X)=C_{\gamma}(X)$.
\end{enumerate}
\end{thm}

\begin{proof}
Let $\mathcal{M}$ be a  model of $\mathbf{ZF}$ such there exists in $\mathcal{M}$ an uncountable Hausdorff space $Y$ which is compact and such that all continuous real functions on $Y$ are constant (see  \cite{GT}, as well as Form  78 in models $\mathcal{N}3$  and $\mathcal{N}8$ in \cite{HR}). In view of Theorem 2.20, there exist in $\mathcal{M}$ a discrete space $D_Y$  and a Hausdorff  compactification $\gamma D_Y$ of $D_Y$ such that  $(\gamma D_Y)\setminus D_Y$ is homeomorphic with $Y$. Put $X=D_Y$ and, for simplicity, assume that $\gamma X\setminus X=Y$. Then it holds true in $\mathcal{M}$ that $\gamma X$ is a strange compactification of $X$. Let $\alpha X$ be the  one-point compactification  in $\mathcal{M}$ of $X$. Of course, $\gamma X$ and $\alpha X$ are non-equivalent, while $C_{\gamma}(X)=C_{\alpha}(X)$. 
\end{proof}

\begin{cor}
It is not a theorem of \textbf{ZF} that if $\alpha X$ and $\gamma X$ are
Hausdorff compactifications of a Tychonoff space $X$ such that $%
C_{\alpha}(X)\subseteq C_{\gamma}(X)$, then $\alpha X\leq \gamma X$.
\end{cor}

One can easily prove the following:

\begin{thm}
$\mathbf{[ZF]}$ If Hausdorff compactifications $\alpha X$ and $\gamma X$ of
a topological space $X$ are both completely regular, then $\alpha
X\leq\gamma X$ if and only if $C_{\alpha}(X)\subseteq C_{\gamma}(X)$.
\end{thm}

\begin{p}
Find in $\mathbf{ZF}$ reasonable necessary and sufficient internal
conditions for a Tychonoff space to have no strange Hausdorff
compactification.
\end{p}

We recall that an \emph{amorphous set} is an infinite set $X$ such that if $%
A $ is an infinite subset of $X$, then the set $X\setminus A$ is finite (see
E.11 in Section 4.1 of \cite{Her}, Form 64 and Note 57 in \cite{HR}).
Amorphous sets exist, for instance, in $\mathbf{ZF}$-model $\mathcal{M}37$ of 
\cite{HR} (see also model $\mathbf{N}1$ of \cite{HR} together with entries (361,64) and (363, 64) on page 335 in \cite{HR}).

\begin{defi}
A topological space $\langle X, \tau\rangle$ will be called \emph{amorphous}
if $X$ is an amorphous set.
\end{defi}

\begin{prop}
$\mathbf{[ZF]}$ Every amorphous Hausdorff space is either discrete or a
one-point Hausdorff compactification of an amorphous discrete space.
\end{prop}

\begin{proof} Let $X$ be an amorphous Hausdorff space. Suppose that $X$ is not discrete. Then $X$ has exactly  one accumulation point. Let $x_0$ be the unique accumulation point of $X$ and let $Y=X\setminus\{x_0\}$. Then $Y$ is a discrete amorphous space such that $X$ is a one-point Hausdorff compactification of $Y$.
\end{proof}

The following theorem, together with Theorem 2.21, points out that a
satisfactory solution to Problem 2.24 can be complicated even for discrete
spaces:

\begin{thm}
$\mathbf{[ZF]}$ Let $X$ be an infinite discrete space. Then every Hausdorff
compactification of $X$ is equivalent with the Alexandroff compactification of 
$X$ if and only if $X$ is amorphous.
\end{thm}

\begin{proof}  Let $\alpha X$ be a Hausdorff compactification of $X$. If $\alpha X\setminus X$ is not a singleton, then there is a pair $x,y$ of distinct points of $\alpha X\setminus X$, so there exists a pair $U,V$ of disjoint open sets in $\alpha X$ such that $x\in U$ and $y\in V$. Then the sets $U\cap X$ and $V\cap X$ are disjoint infinite  subsets of $X$, hence $X$ cannot be amorphous.  On the other hand, if $X$ is not amorphous, then there are disjoint infinite subsets $Y, Z$ of $X$ such that $X=Y\cup Z$, which implies that $X$ has a two-point Hausdorff compactification.  
\end{proof}

\begin{cor}
$\mathbf{[ZF]}$ If a discrete space $X$ is amorphous, then its \v Cech-Stone
compactification is the Alexandroff compactification of $X$.
\end{cor}

\begin{cor}
$\mathbf{[ZF]}$ The following conditions are equivalent:

\begin{enumerate}
\item[(i)] there do not exist amorphous sets;

\item[(ii)] every infinite discrete space has a Hausdorff compactification
whose remainder is not a singleton.
\end{enumerate}
\end{cor}

\begin{thm}
$\mathbf{[ZF]}$ Let $n\in\omega\setminus \{0\}$ and let $X$ be an infinite
discrete space. Then $X$ has an $n$-point Hausdorff compactification and
does not have any $(n+1)$-point Hausdorff compactification if and only if $%
X $ is a disjoint union of $n$ amorphous sets.
\end{thm}

\begin{proof} \emph{Necessity}. First, assume that $X$ has an $n$-point compactification. Then there exists a collection $\{V_i: i\in n\}$ of infinite pairwise disjoint subsets of $X$ such that the set $X\setminus\bigcup_{i\in n} V_i$ is compact. Suppose that there exists $i_0\in n$ such that the set $V_{i_0}$ is not amorphous. Then  there exists an infinite subset $U$ of $V_{i_0}$ such that the set $V_{i_0}\setminus U$  is also  infinite. Since Theorem 6.8 of \cite{Ch} (cf. Theorem 2.1 of \cite{M}) is provable in $\mathbf{ZF}$, we can apply it to showing that $X$ has an $(n+1)$-point Hausdorff compactification. Therefore, if $X$ does not have any $(n+1)$-point Hausdorff compactification, the sets $V_i$  must be amorphous for each $i\in n$.

\emph{Sufficiency}.  Now, assume that $\{X_i: i\in n\}$ is a collection of pairwise disjoint amorphous subsets of $X$ such that $X=\bigcup_{i\in n}X_i$. Then $X$ has an $n$-point Hausdorff compactification $\alpha X$ such that, for each pair $i,j\in n$, the sets $X_i$ and $X_j$ have disjoint closures in $\alpha X$ whenever $i\neq j$.  Let $\gamma X$ be an arbitrary Hausdorff compactification of $X$.  Put $Y=\alpha X\times\gamma X$ and define a mapping $r:X\to Y$ by: $r(x)=\langle\alpha(x), \gamma(x)\rangle$ for each $x\in X$. Then $\langle rX, r\rangle$, where $rX=\text{cl}_Yr(X)$, is a Hausdorff compactification of $X$ such that $\alpha X\leq rX$ and $\gamma X\leq rX$. Let $h:rX\to \alpha X$ be such that $h\circ r=\alpha$. Suppose that there exists $z\in\alpha X\setminus X$ such that $h^{-1}(\{ z\})$ is not a singleton. Then  there exists a collection $\{W_j: j\in n+1\}$ of pairwise disjoint infinite subsets of $X$ such that, for each $j\in n+1$, there exists $i\in n$ such that $W_j\subseteq X_i$. There must exist $i_1\in n$ and a pair $j,k$ of distinct numbers from $n+1$ such that $W_j\subseteq X_{i_1}$ and $W_k\subseteq X_{i_1}$. This is impossible because $X_{i_1}$ is amorphous.  The contradiction obtained shows that the compactifications $\alpha X$ and $rX$ are equivalent. Since $\gamma X\leq rX$ and $rX\setminus r(X)$ consists of exactly $n$ points, we have that $\gamma X\setminus \gamma(X)$ consists of at most $n$ points. This completes the proof.
\end{proof}

\begin{thm}
$\mathbf{[ZF]}$ Let $n\in\omega\setminus\{0\}$ and let $X$ be a discrete
space such that $X$ is a union of $n$ pairwise disjoint amorphous sets. Then
the \v Cech-Stone compactification $\beta X$ of $X$ is the unique (up to the
equivalence) $n$-point Hausdorff compactification of $X$.
\end{thm}

\begin{proof} By Theorem 2.30, $X$ has an $n$-point Hausdorff compactification $\alpha X$. We have shown in the proof to Theorem 2.30 that if $\gamma X$ is a Hausdorff compactification of $X$, then $\gamma X\leq\alpha X$.  This implies that $\beta X\thickapprox \alpha X$. Lemma 6.12 of \cite{Ch} is provable in $\mathbf{ZF}$ and we infer from it and from Theorem 2.30 that all $n$-point Hausdorff compactifications of $X$ are equivalent.
\end{proof}

One of the most important theorems on products of \v Cech-Stone
compactifications is Glicksberg's theorem of $\mathbf{ZFC}$ which asserts
that, for infinite Tychonoff spaces $X$ and $Y$, the Cartesian product $%
\beta X\times \beta Y$ is the \v Cech-Stone compactification of $X\times Y$
if and only if $X\times Y$ is pseudocompact (see \cite{G} and Problem
3.12.21 (c) of \cite{En}). We are going to prove that Glicksberg's theorem
fails in every model of $\mathbf{ZF}$ in which there is an amorphous set. A
well-known fact of $\mathbf{ZFC}$ is that if $X$ is a compact Hausdorff
space, while $Y$ is a pseudocompact Tychonoff space, then the product $%
X\times Y$ is pseudocompact (see Corollary 3.10.27 of \cite{En}).
Unfortunately, the proof to it in \cite{En} is not a proof in $\mathbf{ZF}$.
This is why we show a proof in $\mathbf{ZF}$ to the following helpful lemma:

\begin{lem}
$\mathbf{[ZF]}$ Suppose that $X,Y$ are non-empty  topological spaces such
that $X$ is compact and $Y$ is pseudocompact. Then $X\times Y$ is
pseudocompact.
\end{lem}

\begin{proof} Let $g: X\times Y\to \mathbb{R}$ be a continuous function. Put $f=\vert g\vert$ and define a function $F:Y\to\mathbb{R}$ by:
$$F(y)=\sup\{ f(x,y): x\in X\}.$$
To check that $F$ is continuous, consider any point $y_0\in Y$ and real numbers $a,b$ such that $a<F(y_0)<b$. Let $\mc{U}$ be a collection of all non-empty open sets in $X$ such that if $U\in\mc{U}$, then there exists an open in $Y$ set $G$ such that $y_0\in G$ and $f(U\times G)\subseteq (-\infty, b)$. It follows from the continuity of $f$ that $\mc{U}$ is an open cover of $X$. Since $X$ is compact, there exists a finite subcover $\mc{U}_0$ of $\mc{U}$. For each $U\in\mc{U}_0$, we can choose an open neighbourhood $G(U)$ of $y_0$ such that $f(U\times G(U))\subseteq (-\infty, b)$. Since $a<F(y_0)$, there exists $x_0\in X$ such that $a< f(x_0, y_0)$. By the continuity of $f$, there exists an open neighbourhood $V_0$ of $y_0$ such that $f(\{x_0\}\times V_0)\subseteq(a, +\infty)$. Let $V=V_0\cap\bigcap\{G(U): U\in\mc{U}_0\}$. There exists $U_0\in\mc{U}_0$ such that $x_0\in U_0$. If $y\in V$, then $a< f(x_0, y)\leq F(y)<b$; hence, $F$ is continuous at $y_0$. Since $Y$ is pseudocompact, the function $F$ is bounded. This implies that $f$ is bounded, so $g$ is also bounded.
\end{proof}

\begin{thm}
$\mathbf{[ZF]}$ For every amorphous discrete space $X$, the spaces $X$ and $%
\beta X\times X$ are both pseudocompact, while $\beta X\times\beta X$ is not
the \v Cech-Stone compactification of $\beta X\times X$.
\end{thm}

\begin{proof} Let $X$ be an amorphous discrete space. Consider any $f: X\to \mathbb{R}$. Then $f(X)$ is either finite or amorphous. Since there do not exist amorphous linearly ordered sets,  the  set $f(X)$ is finite. This implies that $X$ is pseudocompact.  By Corollary 2.28, $\beta X$ is a one-point Hausdorff compactification of $X$. In view of Lemma 2.32, the space $\beta X\times X$ is pseudocompact. Now, let $A=\{\langle x,y\rangle\in\beta X\times X: x=y\}$. Then $A$ is clopen in $\beta X\times X$. Let  $f:\beta X\times X\to\{0, 1\}$ be defined by $f(z)=0$ if $z\in A$, while $f(z)=1$ if $z\in(\beta X\times X)\setminus A$. The function $f$ is continuous but it is not continuously extendable over $\beta X\times \beta X$ since the sets $f^{-1}(0)$ and $f^{-1}(1)$ do not have disjoint closures in $\beta X\times \beta X$.  
\end{proof}

\begin{cor}
The following statement is independent of $\mathbf{ZF}$: there exist a
compact Tychonoff space $K$ and a pseudocompact Tychonoff space $X$ such
that $X$ has its \v Cech-Stone compactification, while $K\times\beta X$ is
not the \v Cech-Stone compactification of $K\times X$.
\end{cor}

\section{Compactifications generated by sets of functions}

As we have already informed in section 1, it holds true in every model of $%
\mathbf{ZF+UFT}$ that if $X$ is a Tychonoff space and $F\in\mathcal{E}(X)$,
then the space $e_F X$ is compact; however, $e_{F}X$ can fail to be compact
in a model of $\mathbf{ZF+\neg UFT}$. Therefore, it might be useful to find
necessary and sufficient conditions for $F\in\mathcal{E}(X)$ to have the
property that the space $e_{F}X$ is compact in \textbf{ZF}.

For a compactification $\alpha X$ of $X$ and for a function $f\in C_{\alpha}(X)$,  the unique continuous extension of $f$ over $\alpha X$ is usually denoted by $f^{\alpha}$. For $F\subseteq C_{\alpha}(X)$, we put $F^{\alpha}=\{ f^{\alpha}: f\in F\}$.

\begin{thm}
$\mathbf{[ZF]}$ Suppose that $X$ is a Tychonoff space and that $F\in\mathcal{%
E}(X)$. Then the following conditions are equivalent:

\begin{enumerate}
\item[(i)] $e_{F}X$ is compact;

\item[(ii)] there exists a (not necessarily Hausdorff) compactification $%
\alpha X$ of $X$ such that $F\subseteq C_{\alpha}(X)$.
\end{enumerate}
\end{thm}

\begin{proof}
If $e_{F}X$ is compact, then $F\subseteq C_{e_F}(X)$, so (i) implies (ii). Assume that (ii) holds. Let $\alpha X$ be a compactification of $X$ such that $F\subseteq C_{\alpha}(X)$. We define a mapping $h:\alpha X\to\mathbb{R}^F$ by $h(t)(f)=f^{\alpha}(t)$ for all $t\in\alpha X$ and $f\in F$. Since $e_F(X)\subseteq h(\alpha X)$, it follows from the compactness of $\alpha X$ that $e_{F}X\subseteq h(\alpha X)$. We shall show that $e_{F}X=h(\alpha X)$. To this aim, suppose that $y\in h(\alpha X)\setminus e_{F}X$. Let $t\in\alpha X$ be such that $h(t)=y$. There exist a non-empty finite set $K\subseteq F$ and a positive real number $\varepsilon$, such that if $V_f=\mathbb{R}$ for $f\in F\setminus K$, while $V_f=(y(f)-\varepsilon, y(f)+\varepsilon)$ for $f\in K$, then $e_{F}(X)\cap\prod_{f\in F}V_f=\emptyset$. Then $t\in\bigcap_{f\in K}(f^{\alpha})^{-1}(V_f)$ and, by the density of $X$ in $\alpha X$, there exists $z\in X\cap \bigcap_{f\in K}(f^{\alpha})^{-1}(V_f)$. Then $h(z)\in e_{F}(X)\cap\prod_{f\in F}V_f$ which is impossible. The contradiction obtained implies that $e_{F} X=h(\alpha X)$. In consequence, $e_{F} X$ is compact.
Hence (ii) implies (i).
\end{proof}

\begin{defi}
A Hausdorff completely regular compactification $\gamma X$ of a Tychonoff space $X$ will be
called a \emph{functional  \v Cech-Stone compactification} of $X$ if $%
C_{\gamma}(X)=C^{\ast}(X)$.
\end{defi}

The following proposition is an immediate consequence of Theorem 2.23:

\begin{prop}
$\mathbf{[ZF]}$ If $\gamma_1 X$ and $\gamma_2 X$ are functional \v
Cech-Stone compactifications of $X$, then $\gamma_1 X$ and $\gamma_2 X$ are
equivalent.
\end{prop}

\begin{rem}
Suppose that a Tychonoff space $X$ has a functional \v Cech-Stone
compactification. Since all functional \v Cech-Stone compactifications of $X$
are equivalent, let us denote by $\beta^f X$ an arbitrary functional \v
Cech-Stone compactification of $X$.
\end{rem}

\begin{prop}
$\mathbf{[ZF]}$ Let $X$ be a Tychonoff space and let $F=C^{\ast}(X)$. Then $X$  has its functional \v Cech-Stone
compactification if and only if  the space $e_F X$ is compact. Moreover, if $\beta^f X$ exists, then $\beta^f X\thickapprox e_F X$.
\end{prop}

\begin{proof} Suppose that $\beta^{f} X$ exists. It follows from Theorem 3.1 that the space $e_F X$ is compact. If $e_F X$ is compact, then  since $C_{e_F}(X)=C^{\ast}(X)$, we have $\beta^{f} X\thickapprox e_F X$ by Theorem 2.23. 
\end{proof}

It was observed in \cite{PW} that Taimanov's Theorem 3.2.1 of \cite{En} is
valid in $\mathbf{ZF}$ (see Theorem 5.15 of \cite{PW}). By using Taimanov's
theorem and Lemma 2.10, one can easily prove the following proposition:

\begin{prop}
$\mathbf{[ZF]}$ Suppose that a Tychonoff space X has its functional \v{C}%
ech-Stone compactification. If $K$ is a compact Tychonoff space, then every
continuous mapping $f:X\rightarrow K$ is continuously extendable over $\beta
^{f}X$.
\end{prop}

\begin{cor}
$\mathbf{[ZF]}$ Let $X$ be a Tychonoff space such that $\beta^{f}X$ exists.
Then, for every completely regular Hausdorff compactification $\alpha X$ of $%
X$, we have $\alpha X\leq \beta^{f}X$.
\end{cor}

\begin{thm}
$\mathbf{[ZF]}$ Let $X$ be a non-empty  Tychonoff space such that $\beta X$
exists. Then there exists $\beta ^{f}X$ and $\beta ^{f}X\leq \beta X$.
\end{thm}

\begin{proof} If  $\beta X$ exists, it  follows from Theorem 3.1 that $\beta^{f}X$ exists, too. In view of Remark 1.4, $\beta^{f}X\leq\beta X$. 
\end{proof}

From Theorem 3.8 and Corollary 3.7, we deduce the following:

\begin{prop}
$\mathbf{[ZF]}$ Suppose that $X$ is a Tychonoff space such that $\beta X$
exists. If $\beta X$ is completely regular, then $\beta X\thickapprox
\beta^f X$.
\end{prop}

\begin{thm}
$\mathbf{[ZF]}$ For every non-empty Tychonoff space $X$, the following
conditions are equivalent:

\begin{enumerate}
\item[(i)] the Wallman space $\mathcal{W}(X, \mathcal{Z}(X))$ is compact;

\item[(ii)] there exists a compactification $\alpha X$ of $X$ such that $%
C_{\alpha}(X)=C^{\ast}(X)$;

\item[(iii)] $e_{F}X$ is compact where $F=C^{\ast}(X)$.
\end{enumerate}
\end{thm}

\begin{proof}
It is obvious that (i) implies (ii) and (iii) implies (ii). That (ii) implies (iii) follows from Theorem 3.1. To show that (ii) implies (i),  let us assume (iii) and put $h=e_F$ where $F=C^{\ast}(X)$. Now, consider any filter $\mc{A}$ in $\mc{Z}(X)$. By the compactness of $e_{F}X$, there exists $p\in\bigcap_{A\in\mc{A}}\text{cl}_{e_{F}X}[h(A)]$. We define 
$$\mc{F}=\{ f^{-1}(0): p\in\text{cl}_{e_{F}X}h(f^{-1}(0)): f\in F\}.$$ 
Then $\mc{A}\subseteq\mc{F}$. We shall prove that $\mc{F}$ is a filter in $\mc{Z}(X)$. 

Let $Z_1, Z_2\in\mc{F}$. To show that $Z_1\cap Z_2\in\mc{F}$, suppose that $p\notin\text{cl}_{e_{F}X}[h(Z_1\cap Z_2)]$. By the complete regularity of $\mathbb{R}^F$, there exists $\psi\in C^{\ast}(\mathbb{R}^F)$ such that $\psi(p)=0$ and $h(Z_1\cap Z_2)\subseteq\psi^{-1}(1)$. Let $A_1=Z_1\cap h^{-1}(\psi^{-1}((-\infty,  \frac{1}{2}]))$ and $A_2=Z_2\cap h^{-1}(\psi^{-1}((-\infty, \frac{1}{2}]))$. Then $A_1, A_2\in\mc{Z}(X)$ and $A_1\cap A_2=\emptyset$. There exists $g\in C^{\ast}(X)$ such that $A_1\subseteq g^{-1}(0)$ and $A_2\subseteq g^{-1}(1)$. For the projection $\pi_g:\mathbb{R}^F\to \mathbb{R}$, we have $h(A_1)\subseteq \pi_{g}^{-1}(0)$ and $h(A_2)\subseteq\pi_{g}^{-1}(1)$, so $\text{cl}_{e_{F}X}[h(A_1)]\cap\text{cl}_{e_{F}X}[h(A_2)]=\emptyset$. This contradicts the fact that $p\in \text{cl}_{e_{F}X}[h(A_1)]\cap\text{cl}_{e_{F}X}[h(A_2)]$. Hence $p\in\text{cl}_{e_{F}X}[h(Z_1\cap Z_2)]$. This implies that $Z_1\cap Z_2\in\mc{F}$, so $\mc{F}$ is a filter in $\mc{Z}(X)$. To check that the filter $\mc{F}$ is maximal in $\mc{Z}(X)$, suppose that $\mc{H}$ is a filter in $\mc{Z}(X)$ such that $\mc{F}\subseteq\mc{H}$. Suppose that $Z\in\mc{H}$ and $Z\notin\mc{F}$. Then $p\notin\text{cl}_{e_{F}X}h(Z)$. By the complete regularity of $\mathbb{R}^{F}$, there exists $A\in\mc{Z}(X)$ such that $A\cap Z=\emptyset$ and $p\in\text{cl}_{e_{F}X}(h(A))$. Then $A\in\mc{F}$, so $A\in\mc{H}$.  This is impossible because  $Z\in\mc{H}$, while $Z\cap A=\emptyset$.  Therefore, $\mc{F}$ is an ultrafilter in $\mc{Z}(X)$. Since every filter in $\mc{Z}(X)$ is contained in an ultrafilter in $\mc{Z}(X)$, the Wallman space $\mc{W}(X, \mc{Z}(X))$ is compact. Hence (ii) implies (i).
\end{proof}

\begin{cor}
$\mathbf{[ZF]}$ Let $X$ be a non-empty  Tychonoff space which has its
functional \v{C}ech-Stone compactification. Then $\mathcal{W}(X,\mathcal{Z}%
(X))$ is a Hausdorff compactification of $X$ equivalent with $\beta ^{f}X$.
\end{cor}

\begin{proof} It follows from Theorem 3.10 that $\mc{W}(X, \mc{Z}(X))$ is compact. Since, for every pair $A,B$ of disjoint  sets from $\mc{Z}(X)$, the closures of $A$ and $B$ in $\beta^f X$ are also disjoint, it follows from Theorem 5.15 of \cite{PW} that the mapping $h_{\mc{Z}(X)}: X\to \mc{W}(X, \mc{Z}(X))$ is continuously extendable over $\beta^{f} X$, hence $\mc{W}(X,\mc{Z}(X))\leq\beta^{f}X$. On the other hand, if $A, B\in\mc{Z}(X)$ are disjoint, then the closures of $A$ and $B$ in $\mc{W}(X, \mc{Z}(X))$ are disjoint; therefore, in view of Theorem 5.15 of \cite{PW}, the mapping $\text{id}_X: X\to\beta^{f} X$ has a continuous extension over $\mc{W}(X,\mc{Z}(X))$. This gives that $\beta^{f}X\leq\mc{W}(X, \mc{Z}(X))$. 
\end{proof}

\begin{rem}
For a compactification $\alpha X$ of $X$, let $\mc{Z}_{\alpha}(X)=\{ f^{-1}(0): f\in C_{\alpha}(X)\}$. It may happen in a model of \textbf{ZF} that, for a completely regular
space $X$, there exists a Hausdorff compactification $\alpha X$ of $X$ such
that $\mc{Z}(X)=\mc{Z}_{\alpha}(X)$, while the Wallman space $\mathcal{W}(X, \mc{Z}(X))$ is not compact. To prove this, let us notice that Form 70 of \cite{HR} is equivalent to the statement: There are no free ultrafilters in the power set $\mc{P}(\omega)$. In the model $\mc{M}2$ of \cite{HR}, Form 70 of \cite{HR} is false. This implies that the Wallman space $\mathcal{W}(\omega, \mathcal{P}(\omega))$ is not compact in $\mc{M}2$. However, for the Alexandroff compactification $\alpha_a\omega$ of the discrete space $\omega$ in $\mc{M}2$, we have that $\mathcal{Z}(\omega)=\mathcal{Z}_{\alpha_a}(\omega)$.
\end{rem}

For a topological space $X$, let us denote by $Cl(X)$ the collection of all closed sets of $X$.

\begin{thm} $\mathbf{[ZF]}$ Suppose that $X$ is a $T_1$-space which satisfies $\mathbf{UL}(X)$.  Then the following conditions are equivalent:
\begin{enumerate}
\item[(i)]  the Wallman space $\mc{W}(X, Cl(X))$ is compact;
\item[(ii)] the \v Cech-Stone compactification of $X$ exists;
\item[(iii)] the functional \v Cech-Stone compactification  of $X$ exists.
\end{enumerate}
\end{thm}
\begin{proof} It is obvious that if (i) holds, then the compactification $\mc{W}(X, Cl(X))$ is the \v Cech-Stone compactification of $X$, so (i) implies (ii). In view of Theorem 3.8, (ii) implies (iii).  

Assume that $\beta^{f} X$ exists.  Let $K$ be a compact Hausdorff space and let $h: X\to K$ be a continuous mapping. Consider any pair $A, B$ of disjoint closed sets of $K$. Since $\mathbf{UL}(X)$ holds, the sets $h^{-1}(A)$ and $h^{-1}(B)$ are functionally separated in $X$. In consequence, the sets $h^{-1}(A)$ and $h^{-1}(B)$ have disjoint closures in $\beta^{f}X$. By Theorem 5.15 of \cite{PW}, the function $h$ is continuously extendable over $\beta^{f} X$. This proves that $\beta^{f} X$ is the \v Cech-Stone compactification of $X$. Therefore, (ii) and (iii) are equivalent. 

Assume (ii).  Let $\mc{A}$ be a filter in $Cl(X)$. There exists $p\in\bigcap_{A\in\mc{A}}\text{cl}_{\beta X}A$. We define
$$\mc{F}=\{A\in Cl(X): p\in\text{cl}_{\beta X}A\}.$$ 
Let us check that $\mc{F}$ is a filter in $Cl(X)$. Let $C_1, C_2\in\mc{F}$. Suppose that $p\notin\text{cl}_{\beta X}(C_1\cap C_2)$. In much the same way, as in the proof to Theorem 3.10, we find a set $D\in\mc{Z}(\beta X)$ such that $p\in\text{int}_{\beta X}D$ and $D\cap(C_1\cap C_2)=\emptyset$. We put $A_i=D\cap C_i$ for $i\in\{1,2\}$. We have that $p\in\text{cl}_{\beta X}A_i$ for $i\in\{1,2\}$. On the other hand, since $X$ satisfies $\mathbf{UL}(X)$, the sets $A_1, A_2$ have disjoint closures in $\beta X$. This is impossible. In consequence, $C_1\cap C_2\in\mc{F}$. Therefore, $\mc{F}$ is a filter. Of course, $\mc{A}\subseteq\mc{F}$. We check that $\mc{F}$ is an ultrafilter in $Cl(X)$. To do this, consider any filter $\mc{H}$ in $Cl(X)$ such that $\mc{F}\subseteq\mc{H}$. Suppose that there exists $A\in\mc{H}\setminus\mc{F}$. Then $p\notin\text{cl}_{\beta X} A$, so there exists $Z\in\mc{Z}(\beta X)$ such that $p\in\text{int}_{\beta X} Z$ and $Z\cap A=\emptyset$. It is obvious that $Z\cap X\in\mc{F}$, so $Z\cap X\in\mc{H}$. This is impossible because $A\in\mc{H}$ and $(Z\cap X)\cap A=\emptyset$. Therefore, $\mc{F}$ is an ultrafilter in $Cl(X)$. Since every filter in $Cl(X)$ can be enlarged to an ultrafilter in $Cl(X)$, the space $\mc{W}(X, Cl(X))$ is compact. 
\end{proof}

\begin{cor} $\mathbf{[ZF]}$ Let $X$ be a normal $T_1$-space such that $\beta^{f} X$ exists. If $\mathbf{UL}(X)$ holds, then
$\beta X\thickapprox \beta^{f} X\thickapprox\mc{W}(X, Cl(X))$.
\end{cor}

\begin{rem} Theorem 18 of \cite{HerK} follows directly from our Corollary 3.14 and Theorem 3.10.
\end{rem}

Since $\mathbf{UFT}$ is equivalent to the Boolean Prime Ideal Theorem $%
\mathbf{BPI}$ (Form 14 in \cite{HR}, denoted by \textbf{PIT} in \cite{Her}),
all statements equivalent to $\mathbf{BPI}$ are also equivalent to
conditions (i)-(vii) of the following theorem:

\begin{thm}
$\mathbf{[ZF]}$ The following conditions are equivalent:

\begin{enumerate}
\item[(i)] $\mathbf{UFT}$;

\item[(ii)] every Tychonoff space has its functional \v Cech-Stone
compactification;

\item[(iii)] every Cantor cube $2^{J}$ has a Hausdorff compactification $Y$ such that, for each $j\in J$, the projection $\pi _{j}:2^{J}\rightarrow
\{0,1\}$ has a continuous extension over $Y$;

\item[(iv)] every Cantor cube $2^J$ is compact;

\item[(v)] every Tychonoff space $X$ has a Hausdorff compactification $%
\alpha X$ such that every clopen set of $X$ has a clopen closure in $\alpha
X$;

\item[(vi)] every Tychonoff space $X$ has a Hausdorff compactification $%
\alpha X$ such that every continuous function from $X$ into the discrete
space $\{0,1\}$ is continuously extendable to a function from $\alpha X$
into $\{0, 1\}$;

\item[(vii)] every discrete space $X$ has a compactification $\alpha X$ such
that, for each subset $A$ of $X$, the closure in $\alpha X$ of $A$ is clopen
in $\alpha X$.
\end{enumerate}
\end{thm}

\begin{proof} That (i) implies (ii) follows  Theorem 4.70 of \cite{PW} and from our Theorem 3.1. It is obvious that (ii) implies (iii). When we replace the unit interval $[0, 1]$ by the two-point discrete space $\{ 0,1\}$ and use similar argument as in the proof to Theorem 10.12 in \cite{PW}, we can show that (iii) implies (iv). In view of Theorem 4.70 of \cite{Her}, (iv) and (i) are equivalent. Of course, (v) and (vi) are equivalent,  (ii) implies (vi) and (vi) implies (vii). Suppose that (vii) holds. Let $X$ be a non-empty discrete space and let $f\in C^{\ast}(X)$. Consider any pair $C, D$ of disjoint closed subsets of $\mathbb{R}$. Put  $A=f^{-1}(C)$ and $B=f^{-1}(D)$. Let $\alpha X$ be a compactification of $X$ such that every subset of $X$ has a clopen closure in $\alpha X$. Then the sets $A, B$ must have disjoint closures in $\alpha X$.  It follows from Theorem 5.15 of \cite{PW} that $f$ is continuously extendable over $\alpha X$. It follows from Theorem 3.10 that the Stone space $\mc{S}(X)=\mc{W}(X, \mc{Z}(X))$ is compact. Hence (vii) implies (i) (cf. Remark 2.9 of \cite{PW}).
\end{proof}

By applying our Theorem 3.1 and the proof to Theorem 10.12 of \cite{PW}, we
can immediately deduce the following:

\begin{prop}
$\mathbf{[ZF]}$ Let $J$ be an infinite set. Then the Cantor cube $2^J$ is
compact if and only if there exists a Hausdorff compactification $Y$ of $2^J$
such that, for each $j\in J$, the projection $\pi_j:2^J\to\{0, 1\}$ is
continuously extendable over Y. Similarly, the Tychonoff cube $[0, 1]^J$ is
compact if and only if there exists a Hausdorff compactification $Z$ of $[0,
1]^J$ such that, for each $j\in J$, the projection $\pi_j:[0, 1]^J\to [0, 1]$ is
continuously extendable over $Z$.
\end{prop}

\begin{thm}
The following sentences are relatively consistent with $\mathbf{ZF}$:

\begin{enumerate}
\item[(i)] There exists a 0-dimensional $T_1$-space which has no
compactification $\alpha X$ such that every clopen subset of $X$ has a
clopen closure in $\alpha X$.

\item[(ii)] There exists a metrizable 0-dimensional space which has no
compactification $\alpha X$ such that every clopen subset of $X$ has a
clopen closure in $\alpha X$.
\end{enumerate}
\end{thm}

\begin{proof} It was shown in \cite{W2} that in some models of $\mathbf{ZF}$ (for instance,  in the model $\mc{M}7$ of \cite{HR})  there exists a metrizable non-compact Cantor cube.  Let  $X$ be a metrizable non-compact Cantor cube. Of course, $X$ is 0-dimensional. It follows from Theorem 3.1 and Proposition 3.17 that $X$ does not have a compactification $\alpha X$ such that  every clopen subset of $X$ has a clopen closure in $\alpha X$. Hence both (i) and (ii) are relatively consistent with $\mathbf{ZF}$. 
\end{proof} 

\begin{rem}
It is known that in the model $\mathcal{M}7$ of \cite{HR}, the Cantor cube $2^{\mathbb{R%
}}$ is not compact (see \cite{K1}). Hence, in view of Proposition 3.17 and
Theorem 3.1, $2^{\mathbb{R}}$ is a 0-dimensional $T_1$-space which has no
compactifications in $\mathcal{M}7$ in which closures of clopen sets in $%
2^{\mathbb{R}}$ are clopen. However, by Theorem 2.2 of \cite{W2}, the space $%
2^{\mathbb{R}}$ cannot be metrizable.
\end{rem}

\begin{defi}
Let $X$ be a topological space and $Clop(X)$ the collection of all clopen
subsets of $X$. Every filter in the family $Clop(X)$ will be called a \emph{clopen
filter} of $X$. Every ultrafilter in $Clop(X)$ will be called a \emph{clopen
ultrafilter} of $X$.
\end{defi}

In the light of Theorem 3.18, it may be useful to have a deeper look at
clopen filters.

\begin{thm}
$\mathbf{[ZF]}$ Let $X$ be a topological space. Suppose that $\alpha X$ is a
compactification of $X$ such that every set $A\in Clop(X)$ has a clopen
closure in $\alpha X$. Then every clopen filter of $X$ is included in a
clopen ultrafilter of $X$.
\end{thm}

\begin{proof} Let $\mc{H}$ be a clopen filter of $X$. It follows from the compactness of $\alpha X$ that the set $K=\bigcap_{H\in\mc{H}}\text{cl}_{\alpha X}(H)$ is non-empty. Let us fix $x\in K$ and put 
$$\mc{F}=\{ F\in Clop(X): x\in \text{cl}_{\alpha X}F\}.$$
We show that $\mc{F}$ is a clopen ultrafilter of $X$ such that $\mc{H}\subseteq\mc{F}$. Clearly, $\emptyset\notin\mc{F}$ and $\mc{H}\subseteq \mc{F}$.  For each $A\in Clop(X)$, we have $\alpha X=\text{cl}_{\alpha X}(A)\cup\text{cl}_{\alpha X}(X\setminus A)$. Hence, since $X$ is dense in $\alpha X$, it follows from our hypothesis that, for each $A\in Clop(X)$, the point $x$ belongs to exactly one of the sets $\text{cl}_{\alpha X}(A)$ and $\text{cl}_{\alpha X}(X\setminus A)$, so exactly one of $A$ and $X\setminus A$ belongs to $\mc{F}$. 

Consider any $A,B\in\mc{F}$. We show that $A\cap B\in\mc{F}$. Assume the contrary that $A\cap B\notin\mc{F}$. Then $(X\setminus A)\cup(X\setminus B)\in\mc{F}$ and, consequently, $x\in\text{cl}_{\alpha X}(X\setminus A)$ or $x\in\text{cl}_{\alpha X}(X\setminus B)$. This implies that either $A\notin\mc{F}$ or $B\notin\mc{F}$ -a contradiction. Hence $\mc{F}$ is closed under finite intersections. Of course, for any $A\in\mc{F}$  and $B\in Clop(X)$ such that $A\subseteq B$, we have $B\in\mc{F}$. All this taken together implies that $\mc{F}$ is a clopen ultrafilter of $X$ such that $\mc{H}\subseteq\mc{F}$.
\end{proof}

\begin{rem}
(a) The requirement from Theorem 3.21 that if $A\in Clop(X)$, then $\text{cl}%
_{\alpha X}(A)\in Clop(\alpha X)$, cannot be dropped out even when $X$ is
discrete. Indeed, if $\mathcal{M}$ is a $\mathbf{ZF}$-model such that $%
\omega $ has no free ultrafilters in $\mc{M}$, e.g. Feferman's Model $\mc{M}2$ or Solovay's Model $\mathcal{M}%
5(\aleph) $ in \cite{HR}, then the clopen filter $\mathcal{H}$ of all
cofinite subsets of the discrete space $\omega$ does not extend to a clopen
ultrafilter of $\omega$ in $\mathcal{M}$. The Alexandroff compactification $%
\alpha_a\omega$ of the discrete space $\omega$ is a Hausdorff
compactification of $\omega$ such that, for each infinite subset $A$ of $%
\omega$, if $\omega\setminus A$ is infinite, the closure of $A$ in $%
\alpha_a\omega$ is not clopen.

(b) That conditions (i) and (vii) of Theorem 3.16 are equivalent can be
proved by applying Theorem 3.21. Namely, for a discrete space $X$, the Stone
space $\mathcal{S}(X)$ is compact if and only if every filter in the power
set $\mathcal{P}(X)$ is contained in an ultrafilter in $\mathcal{P}(X)$. By
Theorem 3.21, condition (vii) of Theorem 3.16 implies that $S(X)$ is compact
for every discrete space $X$. Of course, $\mathcal{S}(X)$ is compact for
every discrete space $X$ if and only if $\mathbf{UFT}$ holds.

(c) It was shown in \cite{K2} that the following are equivalent:\newline
(i)\ every clopen filter of $2^{\mathbb{R}}$ extends to a clopen ultrafilter
of $2^{\mathbb{R}}$ (Form 139 in \cite{HR}). \newline
(ii)\ $\mathbf{BPI}(\omega )$ : every filter on $\omega $ extends to an
ultrafilter (Form 225 in \cite{HR}).\newline
(iii) the Cantor cube $2^{\mathbb{R}}$ is compact.\\ Hence, $2^{\mathbb{R}}$
is not compact if and only if  there exists a clopen filter $\mathcal{H}$ of 
$2^{\mathbb{R}}$ which does not extend to a clopen ultrafilter of $2^{%
\mathbb{R}}$. Thus, it follows directly from Theorems 3.1, 3.21 and
Proposition 3.17 that the Cantor cube $2^{\mathbb{R}}$ does not have a
compactification in which closures of clopen sets of $2^{\mathbb{R}}$ are
clopen if and only if there exists a clopen filter of $2^{\mathbb{R}}$ which
does not extend to a clopen ultrafilter of $2^{\mathbb{R}}$.
\end{rem}

\begin{rem}
(a) Let $X$ be a Tychonoff space and $\alpha X$ be a Hausdorff
compactification of $X$. It is known that if $\alpha X$ is a \v{C}ech-Stone
compactification of $X$, then the following condition is satisfied:

\begin{description}
\item[($\ast $) ] for every clopen set $A$ of $X$, $\text{cl}_{\alpha X}(A)$
is a clopen set of $\alpha X$.
\end{description}

On the other hand, the interval $[0, 1]$ is a non-\v{C}ech-Stone
compactification of the open interval $(0, 1)$ with the usual topology, while $(0, 1)$ satisfies trivially $(\ast
)$. Thus, $(\ast )$ does not imply $\alpha X$ is a \v{C}ech-Stone
compactification of $X$.\smallskip 

(b)  For each $n\in\omega$, let $\mc{F}_n=\{A\subseteq \omega: n\in A\}$. Since the set $W=\{\mathcal{F}_{n}:n\in \omega \}$
is dense in $\mathcal{S}(\omega )$, it follows that any compactification of $%
\mathcal{S}(\omega )$ is also a compactification of the discrete space $%
\omega $. In particular, if $\mathcal{S}(\omega )$ is compact, then it is
the \v{C}ech-Stone compactification of $\omega $. Clearly, if in a $%
\mathbf{ZF}$-model $\mc{M}$ there do not exist free ultrafilters in the collection $\mc{P}(\omega)$, then $%
\mathcal{S}(\omega )=W$ is discrete in $\mathcal{M}$, and the Alexandroff
compactification $\alpha _{a}W$ is a Hausdorff compactification of $W$ but, in view of Theorem 2.27, 
$\alpha _{a}W$ is not a \v{C}ech-Stone compactification of $W$ because $W$ is
not amorphous. In the model $\mathcal{N}[\Gamma ]$ of \cite{HKT}, $\mathcal{S%
}(\omega )$ is a dense-in-itself Tychonoff space which is easily seen not to
be locally compact. Hence, the Alexandroff compactification $\alpha _{a}%
\mathcal{S}(\omega )$ of $\mathcal{S}(\omega )$, in contrast to that of $%
\omega $, is not Hausdorff in $\mathcal{N}[\Gamma ]$. Furthermore, since $%
\mathcal{S}(\omega )$ is not compact in $\mathcal{N}[\Gamma ]$, it follows
from Theorem 3.21 that, for every Hausdorff compactification $Y$ of $%
\mathcal{S}(\omega )$ in $\mathcal{N}[\Gamma ]$, there exists a clopen set $A
$ of $\mathcal{S}(\omega )$ such that c$\text{l}_{Y}(A)\cap \text{cl}_{Y}(%
\mathcal{S}(\omega )\setminus A)\neq \emptyset $. We do not know whether the
(ultrafilter compact) space $\mathcal{S}(\omega )$ (see, e.g., \cite{HHK})
has a Hausdorff compactification in \medskip $\mathcal{N}[\Gamma ]$.
\end{rem}

\begin{prop} $\mathbf{[ZF]}$ If a topological space $X$ has a Hausdorff compactification and there exists a set $\mc{K}$ of Hausdorff compactifications of $X$ such that every Hausdorff compactification of $X$ is equivalent with a member of $\mc{K}$ and, moreover, if the space $\prod_{\gamma X\in\mc{K}} \gamma X$ is compact, then $X$ has its \v Cech-Stone compactification.
\end{prop}
\begin{proof} Let $\mc{K}$ be a set of Hausdorff compactifications of  $X$ such that every Hausdorff compactification of $X$ is equivalent with a member of $\mc{K}$. Put $Y=\prod_{\gamma X\in\mc{K}}\gamma X$ and assume that $Y$ is compact. Let $e: X\to Y$ be defined by: $e(x)(\gamma X)=\gamma (x)$ for all $x\in X$ and $\gamma X\in\mc{K}$. Denote by $e X$ the closure in $Y$ of $e(X)$. Then $eX$ is a Hausdorff compactification of $X$. In much the same way, as in the proof to Theorem 3.5.9 of \cite{En}, one can check that $e X$ is the \v Cech-Stone compactification of $X$.
\end{proof}

\begin{prop} $\mathbf{[ZF]}$ Suppose that a topological space $X$ has its \v Cech-Stone compactification. Then there exists a set $\mc{K}$ of Hausdorff compactifications of $X$ such that every Hausdorff compactification of $X$ is equivalent with a member of $\mc{K}$.
\end{prop}
\begin{proof} Let us consider the collection $\mc{R}$ of all collections $\mc{D}$ of pairwise disjoint closed subsets of $\beta X$ such that each $D\in\mc{D}$ is non-empty,  $\beta X\setminus X=\bigcup_{D\in\mc{D}}D$ and the quotient  space $r_{\mc{D}} X$ obtained from $\beta X$ by identifying each set $D\in\mc{D}$ with a point is a Hausdorff compactification of $X$. It follows from $\mathbf{ZF}$ that the class $\mc{K}=\{r_{\mc{D}}X: \mc{D}\in\mc{R}\}$ is a set. Of course, every Hausdorff compatification of $X$ is equivalent with a member of $\mc{K}$. 
\end{proof}

\begin{rem}
In $\mathbf{ZFC}$, for a non-empty  Tychonoff space $X$, the class $\mathcal{%
K}=\{e_{F}X:F\in \mathcal{E}(X)\}$ is a set of Hausdorff compactifications
of $X$ such that every Hausdorff compactification of $X$ is equivalent with a
member of $\mathcal{K}$. However, perhaps, in a model of $\mathbf{ZF}$, it
may happen that every class $\mathcal{K}$ of Hausdorff compactifications of a space $%
X$ such that every Hausdorff compactification of $X$ is equivalent with a
member of $\mathcal{K}$ is a proper class.
\end{rem}

We include a careful proof to the following theorem for completeness: 

\begin{thm} $\mathbf{[ZF]}$ The following conditions are equivalent:
\begin{enumerate}
\item[(i)] $\mathbf{UFT}$;
\item[(ii)] for every topological space $X$ it holds true that if $X$ has a Hausdorff compactification, then $X$  has the \v Cech-Stone compactification;
\item[(iii)] every Tychonoff space has its \v Cech-Stone compactification.
\end{enumerate}
\end{thm}
\begin{proof} Suppose that $\langle X, \tau\rangle$ is a topological space which has a Hausdorff  compactification. Consider the Stone space $\mc{S}(X)$ of the discrete space $\langle X, \mc{P}(X)\rangle$. Assume that (i) holds. Then $\mc{S}(X)$ is the \v Cech-Stone compactification of $\langle X, \mc{P}(X)\rangle$. Consequently, for every Hausdorff compactification $\gamma \langle X, \tau\rangle$ of $\langle X, \tau\rangle$ the mapping $\text{id}_X: X\to\gamma\langle X, \tau\rangle$ is continuously extendable over $\mc{S}(X)$ to a mapping $g_{\gamma}: \mc{S}(X)\to\gamma\langle X,\tau\rangle$. We  consider the equivalence relation $\thickapprox_{\gamma}$ on $\mc{S}(X)$ defined by: $y\thickapprox_{\gamma} z$ if and only if $g_{\gamma}(y)=g_{\gamma}(z)$ for $y,z\in\mc{S}(X)$. Then the space $\gamma\langle X, \tau\rangle$ is homeomorphic with the quotient space $\mc{S}(X)/\thickapprox_{\gamma}$ (see Theorem 2.4.3 of \cite{En}). Since the class of all quotient spaces obtained from $\mc{S}(X)$ is a set, we can deduce from  the scheme of replacement (Axiom 6 on page 10 of \cite{Ku}) that there exists a set $\mc{K}$ of Hausdorff compactifications of $\langle X, \tau\rangle$ such that every Hausdorff compactification of $\langle X, \tau\rangle$ is equivalent with a member of $\mc{K}$. In the light of Proposition 3.24 and Theorem 4.70 of \cite{Her}, we infer that (i) implies (ii). We can deduce from Theorem 3.16 that (ii) implies (i) and that (iii) implies (i). In the light of Theorem 4.70 of \cite{Her},  it follows from Theorem 3.16 and Proposition 3.24 that (i) implies (iii).
\end{proof}

\begin{cor} $\mathbf{[ZF+UFT+UL]}$ Every normal $T_1$-space has its \v Cech-Stone compactification.
\end{cor}
\begin{proof} We assume $\mathbf{ZF+UFT+UL}$. Let $X$ be a normal $T_1$ space. That $X$ is Tychonoff follows from $\mathbf{UL}(X)$. Since $\mathbf{UFT}$ holds, it follows from Theorem 3.27 or from Corollary 3.14 and Theorem 3.16 that $\beta X$ exists. 
\end{proof}

For a non-empty  topological space $X$, let us consider $C^{\ast }(X)$ with
the metric of uniform convergence $\rho _{u}$ defined by the equality $\rho
_{u}(f,g)=\sup \{|f(x)-g(x)|:x\in X\}$ for $f,g\in C^{\ast }(X)$. The
topology $\tau (\rho _{u})$ induced by $\rho _{u}$ is called the \emph{topology of
uniform convergence} in $C^{\ast }(X)$.

\begin{defi}
A set $A\subseteq C^{\ast}(X)$ will be called:

\begin{enumerate}
\item[(i)] \emph{sequentially closed} in $C^{\ast}(X)$ if, for each uniformly
convergent on $X$ sequence $(f_n)$ of functions from $A$, the limit function 
$f=\lim_{n\to +\infty} f_n$ belongs to $A$;

\item[(ii)] \emph{uniformly closed} in $C^{\ast}(X)$ if $A$ is closed with respect
to the topology of uniform convergence in $C^{\ast}(X)$.
\end{enumerate}
\end{defi}

Of course, every uniformly closed subset of $C^{\ast}(X)$ is sequentially
closed. It follows from Theorem 4.54 of \cite{Her} that it may not be true
in a model of $\mathbf{ZF}$ that every sequentially closed subset of $%
C^{\ast}(X)$ is uniformly closed. The following theorem of $\mathbf{ZFC}$
can be deduced immediately from Theorem 2.12 of \cite{W1}:

\begin{thm}
$\mathbf{[ZFC]}$ If a compactification $\gamma X$ of a non-empty 
topological space $X$ is generated by a set $F\in \mathcal{E}(X)$, then $%
C_{\gamma }(X)$ is the smallest (with respect to inclusion) sequentially
closed subalgebra of $C^{\ast }(X)$ which contains $F$ and all constant
functions from $C^{\ast }(X)$.
\end{thm}

It is still unknown whether Theorem 3.30 can be proved in $\mathbf{ZF}$. In
the original proof to Theorem 2.12 in \cite{W1}, the axiom of choice was
used. Theorem 3.30 is a useful tool for investigations of Hausdorff
compactifications in every model of $\mathbf{ZFC}$.

\textbf{Conclusions.} In this article, we have proved a considerable number of theorems on Hausdorff compactifications with the absence of the axiom of choice.  We have posed non-trivial open problems that are of fundamental importance in $\mathbf{ZF}$-theory of Hausdorff compactifications. More research is needed to solve the problems in a not-too-distant future.

\bigskip 

\end{document}